\documentclass[oneside,10pt,final]{amsart}
\makeatletter
\@namedef{subjclassname@2020}{%
  \textup{2020} Mathematics Subject Classification}
\makeatother
\usepackage{tikz-cd}
\usepackage[notcite,notref]{showkeys}
\usepackage[a4paper,width=160mm,top=27mm,bottom=27mm]{geometry}
\usepackage{cases}
\usepackage{amsmath,amssymb}
\usepackage[hidelinks]{hyperref}
\usepackage{verbatim}

\newtheorem{theorem}{Theorem}[section]
\newtheorem{lemma}[theorem]{Lemma}
\newtheorem{definition}[theorem]{Definiton}
\newtheorem{proposition}[theorem]{Proposition}
\newtheorem{corollary}[theorem]{Corollary}

\theoremstyle{definition}
\newtheorem{remx}[theorem]{Remark}
\newenvironment{remark}
  {\pushQED{\qed}\remx}
  {\popQED\endremx}

\newcommand{\R}{\mathbb{R}}
\newcommand{\C}{\mathbb{C}}
\newcommand{\T}{\mathbb{T}}

\newcommand{\abs}[1]{\lvert#1\rvert}

\numberwithin{equation}{section}

\begin{document}
\address{Seynabou Gueye
\newline \indent AIMS-Senegal
\newline \indent KM2 Route de Joal
\newline\indent P.O Box,1418, Mbour-Thi\`{e}s, S\'{e}n\'{e}gal}
\email{gueyenabou94@gmail.com, 
seynabou.gueye@aims-senegal.org}
\address{Filone G. Longmou-Moffo
\newline \indent AIMS-Senegal
\newline \indent KM2 Route de Joal
\newline\indent P.O Box,1418, Mbour-Thi\`{e}s, S\'{e}n\'{e}gal}
\email{longmoumoffofilonegilson@gmail.com, longmou.m.f.gilson@aims-senegal.org}

\address{Mouhamadou Sy
\newline \indent AIMS-Senegal
\newline \indent KM2 Route de Joal
\newline\indent P.O Box,1418, Mbour-Thi\`{e}s, S\'{e}n\'{e}gal}
\email{mouhamadous314@gmail.com, mouhamadou.sy@aims-senegal.org}

\title[Probabilistic GWP for energy supercritical NLS on compact manifolds]{Probabilistic global-wellposedness for the energy-supercritical Schr\"odinger equations on compact manifolds}
\author{Seynabou Gueye, Filone G. Longmou-Moffo and Mouhamadou Sy}

\begin{abstract}
We consider the nonlinear Schr\"odinger equations with a general nonlinearity power in all dimensions. We construct invariant measures concentrated on Sobolev spaces $H^s$ of singular orders, $s\leq\frac{d}{2}$. We prove almost sure global wellposedness and bounds on the growth in time of the solutions via invariant measure arguments. Our setting includes a generic compact Riemannian manifold; we specify the cases of the torus and Zoll manifolds.
\end{abstract}
\bigskip


\maketitle


\section{Introduction}
\subsection{Context}In the present work, we consider the nonlinear Schrodinger equation (NLS)
\begin{equation}\label{Intro_NLS}
    \partial_tu-i(\Delta u-\abs{u}^{2q}u)=0,\quad \quad (u(t,x),t,x)\in \mathbb{C}\times \mathbb{R}\times{M^d}
\end{equation}
supplemented with an initial condition
\begin{equation}\label{Intro_IC}
    u|_{t=0}=u_0\in H^s(M^d)
\end{equation}
where $M^d=(M^d,g)$ is a compact Riemannian manifold of dimension $d$ with metric $g$, and $H^s(M^d)$ is the Sobolev space of order $s\in\R$ defined on $M^d$. The operator $\Delta$ entering the equation is the classical Laplace-Beltrami operator associated to $g$.\\
 When $M^d$ has boundary, which will be assumed to be smooth enough, we will supplement \eqref{Intro_NLS}, \eqref{Intro_IC} with the Dirichlet condition
\begin{equation}
    u|_{\partial M^d}=0.
\end{equation}
Formally, the NLS equation \eqref{Intro_NLS} is invariant under the scaling transformation $u_\lambda(t,x)=\lambda^{\frac{1}{q}} u(\lambda^{2}t,\lambda x)$. One can easily observe that the corresponding scaling invariant $L^2-$based Sobolev space is $H^{s_{q,d}}$, where
\begin{equation*}
    s_{q,d}=\frac{d}{2}-\frac{1}{q}.
\end{equation*}
We will refer this number as the critical exponent.
It describes, from a heuristics, the threshold of regularity above which local wellposedness is expected to occur. Namely, we generally expect local wellposedness of NLS for initial data in $H^s, s>s_{q,d}$.\\
Smooth enough solutions of NLS also enjoy the following two conservation laws
\begin{align}
    M(u) &=\frac{1}{2}\int_{M^d}|u(t,x)|^2dx, &\text{(Mass)}\\
    E(u) &=\frac{1}{2}\int_{M^d}|\nabla u(t,x)|^2dx+\frac{1}{2q+2}\int_{M^d}|u(t,x)|^{2q+2}dx. &\text{(Energy)}
\end{align}
The equation can then be classified as energy subcricitial, energy critical or energy supercritical depending on whether the critical exponent is smaller, equal or greater than the energy exponent $s=1$. This reflects the controllability issue of the potential part of the energy by the kinetic part; in other words, the control of nonlinear effects by linear ones. The energy supercritical regimes correspond to the situation where the nonlinear contributions can no longer be controlled by the linear effects. In this case, the local theory is generally illposed on the energy space. This leads to serious obstructions in the global in time analysis even for higher-order regularities, making the supercritical energy NLS much less understood than the subcritical or critical counterparts. Their global regularity problem remains a very difficult open question. Below, we recall some recent developments regarding this problem. In this paper, we aim to address this issue and to construct well-behaved global solutions for energy supercritical NLS on compact Riemannian manifolds of any dimension, for \textit{singular data} (i.e. below the $H^\frac{d}{2}$-regularity).\\

\subsection{Background and earlier results}

The nonlinear Schrödinger equation arises from diverse physics theories and has applications in areas such as nonlinear optics, plasma physics, quantum mechanics, and fluid dynamics. Its mathematical analysis has a remarkably rich history, driving the development of many analysis tools of independent interest. We are concerned with the following three general questions.
\begin{enumerate}
    \item Local wellposedness
    \item Global regularity
    \item Asymptotic behavior in time
\end{enumerate}

\subsubsection{Local wellposedness} As for local wellposedness, Strichartz estimates are fundamental tools for establishing useful controls allowing a fixed-point strategy. These estimates reflect the dispersive properties of the equation, which are very sensitive to the geometry of the physical space. Let us define the scaling admissibility condition of a pair $(r,p)$ for the Schr\"odinger equation
\begin{align}\label{Intro_admissibility}
    \frac{2}{r} = d \left(\frac{1}{2} - \frac{1}{p}\right).
\end{align}

\begin{itemize}
    \item For Euclidean spaces \(\mathbb{R}^d\), we can combine functional analytic techniques (such as interpolation and duality arguments) with the disperive inequality provided by the linear Schr\"odinger flow to obtain the following Strichartz estimate \cite{strichartz1977restrictions,ginibre1985global,keel1998endpoint}. For the endpoint case, one uses the $TT^*$ argument of Keel-Tao:
\[
\| e^{it\Delta}u \|_{L^r_t(\R;L^p_x)} \leq C \| u \|_{L^2_x}, \quad \text{with} \quad r \geq 2 \quad\text{and $(r,p)\neq (2,\infty)$}\quad\text{satisfying}\ \ \eqref{Intro_admissibility}.
\]

    \item For a generic compact Riemannian manifold, the properties of the linear Schr\"odinger flow are less conspicuous. In particular, the dispersion inequality that was obtained from the properties of Fourier transform in Euclidean spaces is not directly available. However, following Burq-G\'erard-Tzvetkov, one can proceed by localization arguments, using a semiclassical analysis, to obtain a localized version of the dispersion inequality. Then, a reconstruction argument yields an estimate with a loss of a fraction of regularity compared to the Euclidean case. This holds for compact Riemannian manifolds without /or  with boundary \cite{staffilani2002strichartz,burq2004strichartz,ivanovici2010schrodinger,anton2008strichartz,blair2012strichartz}:
    \begin{align}\label{I_S_C}
\| e^{it\Delta}u \|_{L^r_t(I;L^p_x)} \leq C(I) \| u \|_{H^{\frac{1}{r}}_x}, \quad \text{with} \quad |I|<\infty, \quad r \geq 2,\ q<\infty \quad\text{satisfying}\ \ \eqref{Intro_admissibility}.
\end{align}
\item For specific geometries, refined Strichartz type estimates are available (see Section \ref{LWP}).

\end{itemize}
It is worth mentioning the use of multilinear versions of Strichartz estimates in establishing local well-posedness (see, e.g. \cite{bourgain1993fourier,bgtBil,tzvNLS,tzvNLS06,an,htt,syyu2020almost,syyu2022globalfrac}). This is performed within the framework of Bourgain spaces, that appear more suited to analyze dispersive problems than the parabolic spaces $L_t^pL_x^q$, specifically in low regularities. Local well-posedness results for NLS on both Euclidean spaces and compact settings can be found in \cite{cazenave1990cauchy,bourgain1993fourier,burq2004strichartz,bgtBil,herr2013quintic} and references therein.

\subsubsection{Global well posedness}

In the context of energy subcritical local theories posed in  $H^1$, a blow-up criteria can be stated as follows:
\begin{align}
    T<\infty\implies\lim_{t\uparrow T}\|u(t)\|_{H^1}=+\infty.
\end{align}
Hence, energy conservation can be used in an iteration procedure to obtain global regularity in $H^1$. Global $H^1$ solutions were obtained in various settings, including $\R^d$ and compact manifolds \cite{ginibre1985global,brezis1980nonlinear,bourgain1993fourier,burq2003cauchy,bgtBil}. As for regularities $s\in(s_{q,d},1)$, more sophisticated arguments such as the high-low method of Bourgain \cite{bourRef} and the I-method introduced by Colliander-Keel-Staffilani-Takaoka-Tao \cite{ckst} exploit the energy in decomposition/truncation procedures to perform globalization for local solutions living below the energy space $H^1$ (see also \cite{dodson2019global} and references therein).\\
Returning to $H^1$ global regularity, for the energy critical NLS, the time of local existence is not only a function of the size of the data, but also depends on its profile (or, say, the energy accumulation). To treat the possible concentration of the energy density which would lead to finite time blow up, it is required to perform more sophisticated analysis aimed at ruling out such scenario. This was achieved in radial setting in \cite{bourgain1999global,GrllakscrtNLS,tao2004global}, for the general Euclidean case in \cite{CKSTTcrtNLS,ryckman2007global,visan2007defocusing}, on other unbounded manifolds (hyperbolic space) in \cite{ionescu2012global}, and on compact manifolds (torus, sphere-like manifolds) in \cite{htt,ionpaus,herr2013quintic,ptw}.   \\
As for energy-supercritical regimes, a blow-up criterion was obtained for the Euclidean spaces by Killip and Visan \cite{killip2010energy} (see also \cite{kenig2011nondispersive}) at least for $d\geq 5$:
\begin{align}
    T<\infty\implies\lim_{t\uparrow T}\|u(t)\|_{H^{s_{q,d}}}=+\infty.
\end{align}
Recently, Merle-Rapha\"el-Rodnianski-Szeftel showed, in a groundbreaking work, the finite-time explosion of \eqref{Intro_NLS} on $\R^d$ for the regimes $(d,q)\in\{(5,9),(6,5),(8,3),(9,3)\}$ \cite{merle2022blow}. Partial solutions were constructed in $\R^d$ in \cite{beceanu2019large}. We also note the quasiperiodic solutions constructed by Wang in the periodic setting \cite{wang2016energy}. The third author of the present work developed the Inviscid-Infinite dimensional limit (IID limit) approach in the context of the energy supercritical  NLS posed in $\T^d$ and constructed various invariant measures and proved global well-posedness in the corresponding statistical ensembles \cite{sy2021almost}. These measures are supported by regular data, that is, $H^s$ with $s>\frac{d}{2}$. Then the work \cite{syyu2020almost,sy2022global} proved similar results on the unit ball of $\R^d,\ d={3,4,5}$ for singular data, namely on $H^s$ with $s\leq\frac{d}{2}$.
\subsubsection{Long-time behavior}
On Euclidean spaces, global solutions of the NLS equation \eqref{Intro_NLS} are known to exhibit scattering as time goes to infinity. This follows from decay estimates that can be established in this setting (see \cite{tao2006nonlinear} and references therein). Hence in the long run the flow of the NLS acts nearly linearly on the phase space. 

However, on bounded domains, such estimates are no longer available, and the solutions of the NLS equation do not scatter. The generic long-time behavior issue still remains an unsolved case. Nevertheless, intensive work has been done on the question. Some directions have been investigated, such as the weak turbulence behavior studied through the growth of Sobolev norms of the solutions (see e.g. \cite{colliander2010transfer,hani2015modified}); or the recurrence phenomena via invariant measures (see the literature mentioned below concerning probabilistic arguments).

\subsection{Probabilistic approaches to the global regularity problem}
  Probabilistic globalization methods rely on the idea of using invariant measures as a substitute for a conservation law. This is achieved  through very elaborated procedures combining both deterministic and probabilistic/stochastic arguments. Below, we will present these methods in more details since our globalization argument relies on such techniques.

\begin{enumerate}
    \item {\bfseries Gibbs measures} (see e.g. \cite{lebowitz1988statistical,bourgain1994periodic,bourgNLS96,tzvNLS06,tzvNLS,btt18,nahmod2012invariant,deng2024invariant,bringmann2024invariant,bourgain2014invariant}). This approach relies on the construction of a Gibbs measure based on the energy functional
    \begin{align}
        d\mu = \frac{1}{Z}e^{-E(u)}``du"=\frac{1}{Z}e^{-\frac{1}{2q+2}\|u\|_{L^{2q+2}}^{2q+2}}d\nu;
    \end{align}
    where the measure $d\nu$ is interpreted as a Gaussian measure with covariance operator $(-\Delta)^{-1}$ on an appropriate Hilbert space. Once we have that $u\in L^{2q+2}$ almost surely, we can obtain $d\mu$ well-defined (and non-degenerate) via Radon-Nikodym theorem. Here $Z$ is a normalizing constant. The measure $d\nu$ can be seen as the distribution of the Gaussian random variable
    \begin{align}
        \xi(\omega,x)=\sum_n\frac{1}{\sqrt{\lambda_n}}\xi_n(\omega)e_n(x),
    \end{align}
    where $(\xi_n)$ are i.i.d. complex Gaussian random variables and $(\lambda_n,e_n)$ are spectral data of $-\Delta$ on $(M,g)$. This results in that $\xi$ is a $H^s$-valued Gaussian random variable with $s<1-\frac{d}{2}$.
    \\
    A fundamental future of this theory is the construction of an statistical ensemble. From technical point of view, it can be seen as containing data whose local solutions enjoy, individually, some apriori estimate derived from the moment bounds on the invariant measure. Such an apriori estimate allows the iteration of the local theory to obtain global wellposedness. 
    \item {\bfseries The fluctuation-dissipation} (see e.g. \cite{kuksin2004eulerian,kuksin2004randomly,kuksin2012mathematics,sybo,sykg,foldesSy2021invariant,latocca2023construction}). This method consists in coupling the PDE with a ``heat bath'' which can be decomposed into a dissipation term plus a fluctuation one.
    \begin{align}
        \text{LHS of}\ \eqref{Intro_NLS}=\alpha\,\textbf{Dissip}+\sqrt{\alpha}\,\textbf{Fluct}.
    \end{align}
    Here $\alpha\in (0,1)$ is a small `coupling' parameter. The dissipation and the fluctuation are correctly scaled w.r.t. $\alpha$ to produce useful bounds to be used in the inviscid limit ($\alpha\to 0$). The dissipation is modeled by a negative operator and the fluctuations by a stochastic forcing.
    An stationary measure will result from the existence of a Lyapunov functional combined with the Markov property of the new system. The fluctuation-dissipation relation makes the dissipation effects uniformly controlled by the fluctuation ones. This provides bounds on the moments of the measure that are independent of the coupling parameter $\alpha$. An inviscid limit combined with a compactness argument allows to establish global unique solutions through invariant measure consideration.
    
    \item {\bfseries IID limit} ( see e.g. \cite{sy2021almost,syyu2020almost,syyu2022global,foldessy2023almost}). The Gibbs measures approach meets strong limitations when facing supercritical PDEs due to the singularity of its support ($s<1-\frac{d}{2}$); as soon as $d\geq 2$ the data become distributions for which the definition of the nonlinear term is problematic. On the other hand, the fluctuation-dissipation method does not suffer from support regularity issues ; however, it faces a limitation coming from the fact it doesn't include a detailed construction of a statistical ensemble. In addition, the bounds on the solutions obtained via this approach involve relatively weak integrability in time for the relevant space norm.  The latter is usually quadratic, which limits its applicability to high power nonlinearity that characterizes many supercritical regimes. \\ 
    To overcome these difficulties, the IID limit was introduced in \cite{sy2021almost} as a hybrid method that combines fluctuation-dissipation and Gibbs measures approaches. More precisely, it provides a framework where statistical ensembles can be built upon the fluctuation-dissipation measures. This in particular includes generalizations to non Gaussian-like measures of the overall approximation procedure of Bourgain, and a technical design of the``heat bath" to ensure the existence of the necessary large deviation bounds. This is the method that is adopted in the current work, due to the supercritical nature of the problems that are being studied.
\end{enumerate}
\subsection{Problematic and exposition of the main result}

 The globalization in the context of IID-limit heavily relies on a technical step consisting of the design of a suitable dissipation operator. This determines the regularity and dimensions for which the overall approach could work. In Sy (\cite{sy2021almost}), the dissipation operator allowed to treat smooth regularities $s>\frac{d}{2}$.\\
This was developed further for regularities $s\leq \frac{d}{2}$ in Sy and Yu \cite{sy2022global},\cite{syyu2020almost} in dimension $3$ (and $4$ and 5 for the case $q=1$), in the context of the unit ball, and includes the setting of fractional NLS.\\
While the work \cite{sy2022global} included algebraic powers of NLS on the unit ball of $\R^3$, \cite{syyu2020almost} studied only the case $q=1$ in higher dimensions (the classical theory of global well-posedness were studied in $d=3,4,5$).

In this paper, we address the extension of these results in the following directions:

\begin{enumerate}
    \item Extension on the physical space. Our result includes general compact Riemannian manifold settings;
    \item Extension on the dimension. We consider all dimensions;
    \item We include all ranges of singular Sobolev spaces $(s\leq \frac{d}{2})$ for which local wellposedness is established. We globalize these local theories;
    \item We also include all power nonlinearities beyond certain threshold.
\end{enumerate}
\leavevmode\\
This leads us to three problems determined following the structure of the available Strichartz estimate. Based on three relevant characteristics: type of manifold ($M^d$), the nonlinearity power of the equation $(q)$, and the Sobolev regularity order of the initial data $H^s$:
\begin{align}\label{tableau}
\begin{array}{|c||c|c|c|c|}
\hline
\text{\bfseries } & \text{Riemannian Manifold} & \text{Nonlinearity order } q\geq & \text{Regularity of the initial data}~~s> \\
\hline
\text{Problem 1.} & \text{General compact} &  1 &  \frac{d}{2}-\frac{1}{2q} \left(=s_{q,d}+\frac{1}{2q}\right) \\
\hline
\text{Problem 2.} & \text{Tori} & 1+\frac{2}{d} &s_{q,d} \\
\hline
\text{Problem 3.} & \text{Zoll manifolds}&  2 & s_{q,d}\\
\hline
 \text{Notation} & M^d & =:q_{M^d} & =:s_{M^d}\\
\hline
\end{array}
\end{align}

Here is our main result.
\begin{theorem}\label{recap}
Let $(M^d,g)$ be a compact Riemannian manifold of dimension $d\geq 3$.\\ Let $s\in (s_{M^d},\frac{d}{2}]$~~ and $q\geq q_{M^d}$. For all $ \epsilon>0,$ there are a set $\Sigma=\Sigma_{q,s,\epsilon}\subset H^s$ and a probability measure $\mu=\mu_{q,s,\epsilon}$ such that :
    \begin{enumerate}
        \item There is a global flow $\phi^t$ for NLS  on $\Sigma$, and $ \phi^t\Sigma=\Sigma$ for all $t$;\\ 
        \item For all $u_0\in\Sigma$, we have the bound
        \begin{align}
            \|\phi^tu_0\|_{H^{s^-}}\leq C(\|u_0\|_{H^{s}})(1+|t|)^{\epsilon}\quad \forall t\in \mathbb{R};
        \end{align}
        \item $\mu(\Sigma)=1$;
        \item The measure $\mu$ is invariant under $\phi^t$;
        \item The distribution of functional $u\mapsto\|u\|_{L^2}$ under $\mu$ is absolutely continuous w.r.t. the Lebesgue measure on $\R$.
    \end{enumerate} 
\end{theorem}

\begin{remark}
Notice that for tori and Zoll manifolds we can cover all regularities above $s_{q,d}$, which is not the case for the general manifold. This is due to Strichartz estimates obtained in Bourgain \cite{bourgain1993fourier} (for the torus) and Burq-G\'erard-Tzvetkov \cite{burq2004strichartz} (for Zoll manifolds). 
To simplify the presentation, we opted for establishing a local wellposedness theory relying only on linear Strichartz estimates. This leads to the restriction on the nonlinearity power $q_{M^d}$. However, using a multilinear analysis within the framework of Bourgain spaces, one could push the local wellposedness theory for the lower powers (see e.g. \cite{bourgain1993fourier,bgtBil}), and then globalize the solutions by employing our argument. We will not pursue this direction here since higher powers are our main target. 
\end{remark}

\begin{remark}
    Our proof includes also a noticeable simplification of the infinite dimensional limit step in the IID limit approach (and possibly the corresponding step in the usual Gibbs measure approach). We bypass the use of restriction measures by employing a new and shorter argument in the proof of Lemma \ref{Lemma_full_mu} on the estimation of the size of the statistical ensemble. Beyond this, several technical modifications in this step were necessary due to the power nature of our large deviation estimate.
\end{remark}

\subsection{Organization of the paper}
In Section \ref{LWP}, we present a general local wellposedness theory of \eqref{Intro_NLS} and determine the LWP regularity ranges corresponding to the considered manifolds. This theory is based on Strichartz estimates. We establish a local convergence lemma of Galerkin projections of the equation and formulate a globalization lemma. Section \ref{Sect. Dissip_apriori} is devoted to the anaysis of dissipation operators and a priori estimates that will be useful in the globalization arguments. In Section \ref{Sect_Glob}, we construct invariant measures for the damped-driven Galerkin approximation. We perform the inviscid limit, and construct the statistical ensemble. We present the closing arguments of the proof of the main theorem.  
\subsection{Acknowledgement} The research of M. Sy is funded by the Alexander von Humboldt foundation under the “German Research Chair programme” financed by the Federal Ministry of Education and Research (BMBF). He also thanks the German Academic Exchange Service (DAAD) for supporting his research at AIMS Senegal. S. Gueye and F. G. Longmou-Moffo are grateful to the DAAD for funding their ongoing doctoral research at AIMS Senegal. The authors express their gratitude to Bielefeld University for support as part of a research cooperation that includes hosting programs.

\subsection{General notations}
\leavevmode\par
We denote by $M:=(M,g)$ be a Riemannian compact manifold of dimension $d $, $\Delta$ be the Laplace Beltrami operator on $M$ and $(e_n)_{n\in \mathbb{N}}$ a normalized eigenfunctions basis of $-\Delta$, orthonormal in $L^2(M)$. The associated eigenvalues are noted $(\lambda_n)_{n\in \mathbb{N}}$ such that \\$0=\lambda_0 <\lambda_1\leq \lambda_2 \leq .....\leq \lambda_n \leq ...$
Therefore for $u\in L^2(M)$, we have $$u(x)=\sum_{n\in \mathbb{N}}u_ne_n(x).$$
By Parseval identity, $$\lVert u \rVert^2_{L^2}:=\int_{M}{\abs{u(x)}^2 dx}=\sum_{n\in \mathbb{N}}\abs{u_n}^2.$$
We denote by $\lVert~~ \rVert_{L^p}$ the norm of the Lebesgue space $L^p(M),\ p\in[1,\infty]$, $\lVert \quad \rVert_{H^s}$ the norm of Sobolev space $H^s(M),\ s\in\R$ and and $|~|$ denotes the modulus of a complex number.\\
The Sobolev norm of $H^s$ is defined by $$\lVert u \rVert_{H^s}:=\sqrt{\abs{(1-\Delta)^{\frac{s}{2}}u}_2^2}=\sqrt{\sum_{n\in \mathbb{N}}(1+\lambda_n)^s\abs{u_n}^2}$$\\
We define a real inner product on $L^2$ by $$(u,v) = \text{Re}\left( \int_M u(x) \overline{v(x)} \, dx \right).$$ \\ We have the important property $(u,iu)=0$.\\\\
We denote by $E_N$ the subspace of $L^2$ generated by the finitely many family $\{e_n ;n\leq N\}$, the operator $P_N$ is the projector onto $E_N$ and $E_{\infty}$ refers to $L^2$, and the projector $P_{\infty}$ to the identity operator.\\\\
For a Banach space \( X \) and an interval \( I \subset \mathbb{R} \), we denote by \( C(I,X) = C_I^tX \) the space of continuous functions \( f : I \to X \). The corresponding norm is 
\[
\|f\|_{C^t_IX} = \sup_{t \in I} \|f(t)\|_X.
\]
Let $T>0$, we define by $X^s_T:=C((-T,T);H^s)$ the set of continuous functions \( f : (-T,T) \to H^s \).\\\\
For \( p \in [1,\infty) \), we also denote by \( L^p(I,X) = L^{p}_IX \), the spaces given by the norm
\[
\|f\|_{L^{p}_IX} = \left( \int_I \|f(t)\|_X^p \, dt \right)^{\frac{1}{p}}.
\]

Let $T>0, r,\sigma,p,s$ all positives numbers, we define $$Y_T^s:=C((-T,T);H^s)\cap L^r({(-T,T)};W^{\sigma,p}_x).$$
We have omitted to mention the dependence in $r,\ \sigma,\ p$ in the notation since they will be of secondary interest in our analysis. However, we will precise their range later, in the local theory. The corresponding norm is given by
 $$\rVert u \lVert_{Y_T^s}=\left(\sup_{t\in (-T,T)}\rVert u(t)\lVert_{H^s}^2+\left(\int_0^T\rVert u(\tau)\lVert_{W^{\sigma,p}}^rd\tau)\right)^{\frac{2}{r}}\right)^{\frac{1}{2}}.$$\\
The notation $B_R(X)$ refers to the closed ball with center $0$ and radius $R>0$ of the Banach space $X$.\\
The inequality \( A \lesssim B \) between two positive quantities \( A \) and \( B \) means \( A \leq C B \) for some \( C > 0 \).\\
We denote by $s^- = s - \varepsilon, \text{ for } \varepsilon > 0 \text{ close enough to } 0 \ (\text{we use } s^+ \text{ in a similar way}).$\\
 We denote by $\lfloor ~\rfloor$ the integer part of a number.
\\\\

\section{Local wellposedness} \label{LWP}
This section is devoted to establishing local wellposedness for NLS and to establish deterministic devices that are helpful in the globalization procedure. Here are the results:
\begin{enumerate}
    \item Establishing a general local wellposedness result based on a hypothetical Strichartz type estimate for NLS;
    \item Applying the general LWP for cases where Strichartz estimates are known. Hence the regularity ranges will be concretly determined;
    \item We highlight the uniformity of some quantities, such as the time of size increment, with respect to approximation parameter, such as $N$. This prepares us for the next step
    \item Proving local convergence and globalization lemmas.
\end{enumerate}

\subsection{Local well posedness}
Let us consider the Galerkin projected NLS
\begin{equation}\label{1}
\left\{
\begin{aligned}
\partial_t u &= i\left[ \Delta P_N u - P_N \left( |P_N u|^{2q} P_N u \right) \right],  \\
u(0) &= P_N u_0 \in E^N.
\end{aligned}
\right.
\end{equation}
where $1\leq N\leq\infty$. The following proposition establishes a general local well-posedness result conditioned on availability of a Strichartz estimate. In subsequent corollaries, we will specify the result case by case.
\begin{proposition}\label{local well posedness}
     Let $r,p, \delta$ be positive real numbers with $r>\max(2,2q)$ and $p\geq 1$. Assume that, for any finite time interval $I$, we have\\
     \begin{equation}\label{modified Strichartz estimate}
     \lVert e^{it\Delta}v_0\rVert_{L^r(I;L^p(M))}\leq C(I)\lVert v_0\rVert_{H^\delta}.
     \end{equation}
     \\Let $s>\frac{d}{p}+\delta$ $\text{and} ~~  \sigma=s-\delta$. Then we have the following
     \begin{itemize}
     \item (LWP) For any $R>0$, there is $T:=T(R,s)>0$ such that for any $N\in \mathbb{N}\cup \{\infty\}$, any $u_0 \in B_R(H^s)$, there is a unique $u_N \in Y_T^s$ satisfying (\ref{1}), where $$Y_T^s:=C((-T,T);H^s)\cap L^r((-T,T);W^{\sigma,p}_x).$$\\
     \item (Increment property) Moreover, the existence time $T>0$ can be chosen such that for some independent constant $C>0$, for any $u_0\in H^s$, we have that
     \[\lVert u_N \rVert_{Y^s_T} \leq 2C\lVert P_Nu_0 \rVert_{H^s}.\]\\
     \end{itemize}
\end{proposition}
   Before we present the proof of the proposition \ref{local well posedness}, let us establish the following inhomogeneous Strichartz type inequality based on \eqref{modified Strichartz estimate}.
   \begin{lemma}
   Assume \eqref{modified Strichartz estimate}, then we have, for any $f\in L^1([0,T];H^s),$
       \begin{align}\label{Strichartz integral}
\left\| \int_0^t e^{i(t-\tau)\Delta} f(\tau) \, d\tau \right\|_{L^r([0,T], W^{\sigma,p}(M))} 
 \leq C_T \| f \|_{L^1([0,T], H^s(M))}.
\end{align}
   \end{lemma}
 \begin{proof}
     Following \cite{burq2004strichartz}, let us set $F_\tau=\mathbf{1}_{\tau \leq t}e^{i(t-\tau)\Delta} f(\tau)$ and $J=\left\| \int_0^t e^{i(t-\tau)\Delta} f(\tau) \, d\tau \right\|_{L^r([0,T], W^{\sigma,p}(M))}. $\\
     We have that
$$
J\leq \int_0^T \lVert F_\tau\rVert_{L^r([0,T])W^{\sigma,p}(M)}d\tau   \leq \int_0^T\lVert e^{i(t-\tau)\Delta}f(\tau)\rVert_{L^r([0,T];W^{\sigma,p}(M)}d\tau.
$$
    In fact we have $$\left\lVert \int_0^t e^{i(t-\tau)\Delta}f(\tau)d\tau\right\rVert_{W^{\sigma,p}}^r \leq \left(\int_0^t\lVert e^{i(t-\tau)\Delta}f(\tau)\rVert_{W^{\sigma,p}} ~d\tau\right)^r. $$    \\
    Then $$\left(\int_0^T \left\lVert \int_0^t e^{i(t-\tau)\Delta}f(\tau)d\tau\right\rVert_{W^{\sigma,p}}^rdt\right)^{\frac{1}{r}}\leq \left(\int_0^T \left(\int_0^T\lVert e^{i(t-\tau)\Delta}f(\tau)\rVert_{W^{\sigma,p}}d\tau\right)^r dt\right)^\frac{1}{r}.$$
    By using now the Minkowski inequality on the right term, we obtain $$J\leq \int_0^T\lVert e^{i(t-\tau)\Delta}f(\tau)\rVert_{L^r([0,T];W^{\sigma,p}(M)}d\tau.$$
         Now by using the Strichartz estimate (\ref{modified Strichartz estimate}) and the isometry property, we have\\ $\lVert e^{i(t-\tau)\Delta} f(\tau)\rVert_{L^r([0,T];L^p(M)}=\lVert e^{it\Delta} e^{-i\tau\Delta}f(\tau)\rVert_{L^r([0,T];L^p(M)}\leq C(T)\lVert f(\tau)\rVert_{H^\delta}$ and since \\$\sigma=s-\delta$, we see easily that  $\lVert e^{i(t-\tau)\Delta} f(\tau)\rVert_{L^r([0,T];W^{\sigma,p}(M)}\leq C(T)\lVert f(\tau)\rVert_{H^s}$ 
        and we then obtain $$J\leq C_T \int_0^T \lVert f(\tau)\rVert_{H^s}d\tau =C_T\lVert f\rVert_{L^1([0,T], H^s(M))}.$$
         We arrive at the claim.
 \end{proof}  
\begin{proof}[Proof of Proposition \ref{local well posedness}]
We use the Duhamel principle. The contraction argument relies on the use of the crucial bound \eqref{modified Strichartz estimate}.
      Let us fix $u_0 \in B_R(H^s)$, and set the map
      \[
u\mapsto F(u) = e^{it\Delta } P_N u_0 - i \int_0^t e^{i(t - \tau)\Delta} P_N \left( |P_N u|^{2q} P_N u \right) d\tau\,;\quad Y_T^s\to Y_T^s.
\] By standard product inequality, one has 
\[
\lVert F(u) \rVert_{X^s_T} \leq \lVert P_N u_0 \rVert_{H^s} + C_1 \int_0^T \lVert P_N u \rVert_{L^\infty}^{2q} \lVert P_N u \rVert_{H^s} \, d\tau.
\]
Since we assumed $s>\frac{d}{p}+\delta$, we have, for $\sigma=s-\delta$, that $W^{\sigma,p} \hookrightarrow L^{\infty}$. Therefore, we obtain
$$\lVert F(u) \rVert_{X^s_T} \leq \lVert P_N u_0 \rVert_{H^s} + C_2 \lVert P_N u \rVert_{X^s_T}\int_0^T \lVert P_N u \rVert_{W^{\sigma,p}}^{2q} \, d\tau.
$$\\
According to Holder's inequality, we have, by using the fact that $r>2q$,
\[
\begin{array}{cccc}
\lVert F(u) \rVert_{X^s_T} &\leq& \lVert P_N u_0 \rVert_{H^s} + C_2 T^{(1-\frac{2q}{r})} \lVert P_N u \rVert_{X^s_T} \left( \int_0^T \lVert P_N u \rVert_{W^{\sigma,p}}^{r} \, d\tau \right)^{\frac{2q}{r}}\\\\
&\leq& \lVert P_N u_0 \rVert_{H^s} + C_2 T^{(1-\frac{2q}{r})} \lVert P_N u \rVert_{X^s_T} \lVert P_N u \rVert_{L^r_T W^{\sigma,p}}^{2q}\\\\
&\leq& \lVert P_N u_0 \rVert_{H^s} + C_2 T^{(1-\frac{2q}{r})} \lVert P_N u \rVert_{Y^s_T}^{2q+1}.
\end{array}
\]
Therefore
$\lVert F(u) \rVert_{X^s_T}\leq C_3(\lVert u_0 \rVert_{H^s} + T^{\gamma} \lVert u \rVert_{Y^s_T}^{2q+1})
$ where $\gamma=1-\frac{2q}{r}$, $C_3=\max(1,C_2).$\\
So, we have by using \eqref{modified Strichartz estimate} and \eqref{Strichartz integral}, we have 
\begin{align*}
\lVert F(u) \rVert_{L^r_T W^{\sigma,p}}&\leq C_T\lVert P_N u_0 \rVert_{H^s} + \left\lVert  \int_0^t e^{i(t - \tau)\Delta} P_N \left( |P_N u|^{2q} P_N u \right) d\tau \right\rVert_{L^r_T W^{\sigma,p}}\\
&\leq C_T\left(\lVert P_N u_0 \rVert_{H^s}+   \int_0^T {\lVert  \left( |P_N u|^{2q} P_N u \right)} \rVert_{H^s}d\tau \right)\\
&\leq C'_T\left(\lVert u_0 \rVert_{H^s} +  T^{\gamma} \lVert u \rVert_{Y^s_T}^{2q+1}\right).
\end{align*}

Therefore, we arrive at the following estimate
\[\lVert F(u) \rVert_{Y^s_T}\leq C\left(\lVert u_0 \rVert_{H^s} + T^{\gamma} \lVert u \rVert_{Y^s_T}^{2q+1}\right).\]
Since $u_0\in B_R(H^s)$, we remark that if $$u\in B_{2RC}(Y_T^s) ~~\text{and}~~ T^{\gamma}\leq \frac{1}{2^{2q+2}R^{2q}c^{2q+1}}$$ \\where $c=\max(C,C_6)$, with $C_6$ is some absolute constant,
we have that $F(u)\in B_{2RC}(Y_T^s)$. So, the map $u\mapsto F(u)$ transforms $B_{2RC}$ into itself.\\
Now let $u_1$ and $u_2$ be two elements of $B_{2RC}(Y_T^s )$. Repeating the arguments above, we see that
\begin{align*}
\lVert F(u_1) - F(u_2) \rVert_{X^s_T} &\leq \int_0^T \lVert |P_N u_1|^{2q} P_N u_1 - |P_N u_2|^{2q} P_N u_2 \rVert_{H^s} \, d\tau \\
&\leq C_4 \int_0^T \left( \lVert P_N u_1 \rVert_{L^\infty}^{2q} + \lVert P_N u_2 \rVert_{L^\infty}^{2q} \right) \lVert P_N u_1 - P_N u_2 \rVert_{H^s} \, d\tau\\
 &\leq   C_4 \lVert u_1-u_2 \rVert_{Y^s_T}\left(\int_0^T \lVert  u_1 \rVert_{W^{\sigma,p}}^{2q} \, d\tau +\int_0^T \lVert  u_2 \rVert_{W^{\sigma,p}}^{2q} \, d\tau\right)\\
 &\leq   C_4 \lVert u_1-u_2 \rVert_{Y^s_T} T^{\gamma}(\lVert u_1\rVert_{Y^s_T}^{2q}+\lVert u_2\rVert_{Y^s_T}^{2q}).
\end{align*}
On the other hand, by using (\ref{modified Strichartz estimate} - \ref{Strichartz integral}), we have \\
\begin{align*}
\lVert F(u_1) - F(u_2) \rVert_{L_T^rW^{\sigma,p}} &\leq C_T\int_0^T \lVert |P_N u_1|^{2q} P_N u_1 - |P_N u_2|^{2q} P_N u_2 \rVert_{H^s} \, d\tau \\
&\leq   C_5 \lVert u_1-u_2 \rVert_{Y^s_T} T^{\gamma}(\lVert u_1\rVert_{Y^s_T}^{2q}+\lVert u_2\rVert_{Y^s_T}^{2q}).
\end{align*}
Then, $$\lVert F(u_1) - F(u_2) \rVert_{Y^s_T} \leq C_6 \lVert u_1-u_2 \rVert_{Y^s_T} T^{\gamma}(\lVert u_1\rVert_{Y^s_T}^{2q}+\lVert u_2\rVert_{Y^s_T}^{2q}),$$
with $T^\gamma$ definded as above.
Therefore, we have
\[\lVert F(u_1)-F(u_2)\rVert_{Y^s_T}\leq \frac{1}{2}\lVert u_1-u_2\rVert_{Y^s_T}.\]
This finishes the contraction argument.\\
Now to show the last claim, let us observe that, the solution stay in $B_{2RC}(Y^s_T) ~~\forall \abs{t}<T$. \\
Therefore, we have by using the Duhamel's formula,\\
\begin{align*}
    \lVert u_N \rVert_{Y^s_T}&\leq C\left(\lVert P_Nu_0\rVert_{H^s}+T^{\gamma}\lVert u_N \rVert^{2q+1}_{Y^s_T}\right)\\
    &\leq C\left(\lVert P_Nu_0\rVert_{H^s}+\frac{1}{2C}\lVert u_N \rVert_{Y^s_T}\right).
\end{align*}
Therefore, $$\lVert u_N \rVert_{Y^s_T}\leq 2C\lVert P_Nu_0 \rVert_{H^s}.$$ This finishes the proof.
\end{proof}
We remark that the local existence constructed above does not depend on $N$, this is an important property that will be exploited later. Let us now turn to applications.\\ Our first corollary is generic and is based on the Strichartz estimates on compact manifolds, obtained in \cite{burq2004strichartz,blair2012strichartz}.

\begin{corollary}[General compact Riemannian manifold $M^d$]
    \leavevmode\\
    Let $u_0 \in H^s(M^d)$, then the equation (\ref{1}) is locally well-posed for $s>\frac{d}{2}-\frac{1}{2q}=s_{q,d}+\frac{1}{2q}$.
        
    \end{corollary}
\begin{proof}
    Let us recall the Strichartz estimate with the loss of $\frac{1}{r}$ as seen in \eqref{I_S_C}. We have for all numbers $r,p$ with $r\geq 2$, $p<\infty$, and satisfying the scaling condition $\frac{2}{r}+\frac{d}{p}=\frac{d}{2}$ :
     \begin{align}
       \lVert e^{it\Delta}P_Nu_0 \rVert_{L^r_t([-T,T];L^p(M^d))}\leq C(T) \lVert P_Nu_0\rVert_{H^\frac{1}{r}}   
    \end{align}
    Therefore by taking $r=2q+\epsilon$, since $q\geq 1$, then $r>\max(2,2q)$, we will have according to the Proposition (\ref{local well posedness}), the local well-posed for $s>\frac{d}{p}+\frac{1}{r}=\frac{d}{2}-\frac{2}{r}+\frac{1}{r}=\frac{d}{2}-\frac{1}{2q}=s_{q,d}+\frac{1}{2q}.$
\end{proof}
The following two corollaries concern the specific geometries of the torus and Zoll manifolds, where  the Schr\"odinger propagator enjoys bilinear properties that enabled to improve the generic Strichartz estimate used in the above result (See \cite{bourgain1993fourier,burq2004strichartz}). This allows to cover the regularities all the way to the critical threshold.
\begin{corollary}[Case of the Torus $\mathbb{T}^d, d\geq 3$]
\leavevmode\\
Let $u_0 \in H^s(\T^d)$, then the equation (\ref{1}) is locally well-posed for $s>s_{q,d}=\frac{d}{2}-\frac{1}{q}$ with $q\geq 1+\frac{4}{2d}$.
    
\end{corollary}
\begin{proof}
\leavevmode
    Let us recall first the periodic Strichartz inequalities, we refer to \cite{bourgain1993fourier,bourgain2015proof}
    \begin{equation}{\label{periodic Strichartz}}
        \lVert e^{it\Delta}P_Nu_0\rVert_{L^r(\mathbb{T}^{d+1})}\lesssim N^{\frac{d}{2}-\frac{d+2}{r}+\epsilon}\lVert P_Nu_0\rVert_{L^2} 
    \ \ \ \text{for} \ r\geq\frac{2(d+2)}{d}
    \end{equation}
    Now by using the  Littlewood-Paley inequality, we will have 
    \begin{equation}{\label{periodic Strichartz2}}
        \lVert e^{it\Delta}u_0\rVert_{L^r(\mathbb{T}^{d+1})}\lesssim \lVert u_0\rVert_{H^{\frac{d}{2}-\frac{d+2}{r}+\epsilon}} 
    \ \ \ \text{for} \ r\geq\frac{2(d+2)}{d}
    \end{equation}
    Therefore we can see that by taking $r=2q+\epsilon$, we will have for $q\geq \frac{2(d+2)}{2d}=1+\frac{2}{d}$; \\$r>\max(2,2q)$ and $r\geq\frac{2(d+2)}{d}$; therefore, according to the estimate \eqref{periodic Strichartz2} and Proposition \ref{local well posedness}, we obtain local well-posedness for $$s>\frac{d}{r}+\frac{d}{2}-\frac{d}{r}-\frac{2}{2q}=\frac{d}{2}-\frac{1}{q}=s_{q,d}.$$
    
\end{proof}
\begin{corollary}[Case of Zoll manifold $M^d$ of dimension $d\geq 3 $]
\leavevmode\\
Let $u_0 \in H^s(M^d)$, where $M^d$ stands for a Zoll manifold, then the equation (\ref{1}) is locally well-posed for $s>s_{q,d}=\frac{d}{2}-\frac{1}{q}$ with $q\geq 2$.
\end{corollary}

    \begin{proof}
    Let us recall the Strichartz estimate on the Zoll manifold of dimension $d\geq 3$ (obtained in \cite{burq2003cauchy}):\\
    \begin{align}
      \lVert e^{it\Delta}P_Nu_0 \rVert_{L^4({M}^d\times[-T,T])}\lesssim N^{\frac{d}{4}-\frac{1}{2}+\epsilon}\lVert P_Nu_0\rVert_{L^2}.  
    \end{align}
    We have also by using the $L^2$-norm conservation and the fact that $H^{\frac{d}{2}+}\hookrightarrow L^{\infty}$\\
    \begin{align}
       \lVert e^{it\Delta}P_Nu_0 \rVert_{L^{\infty}({M}^d\times[-T,T])}\lesssim N^{\frac{d}{2}+}\lVert P_Nu_0\rVert_{L^2}.   
    \end{align}
    Now using again the Riesz-Thorin interpolation by taking $\theta=\frac{4}{r}$ with $r> 4$, we have\\
    \begin{align}\label{Strichartz4}
       \lVert e^{it\Delta}P_Nu_0 \rVert_{L^r({M}^d\times[-T,T])}\lesssim N^{\frac{d}{2}(1-\frac{4}{r})+\frac{4}{r}(\frac{d}{4}-\frac{1}{2})+}\lVert P_Nu_0\rVert_{L^2}.  
    \end{align}
    Therefore by taking $r=2q+\epsilon$, for $q\geq 2$, then $r> 4$, so we have according the estimate (\ref{Strichartz4}) and the Proposition \ref{local well posedness}, the local well-posed for $s>\frac{d}{r}+\frac{d}{2}-\frac{2d}{r}+\frac{d}{r}-\frac{2}{2q}=\frac{d}{2}-\frac{1}{q}=s_{q,d}$.
    \end{proof}
\begin{remark}
Notice that for $N$ finite, the local solutions of \eqref{1} can be extended globally in time due to preservation of the $L^2$-norm (and the equivalence of norms).
    Let us define the associated global flow by
\begin{align*}
\phi_N^t : E_N &\to E_N\\
P_Nu_0 &\mapsto \phi_N^t P_Nu_0
\end{align*}
where \(\phi_N^t(P_Nu_0) := u(t,u_0)\) represents the solution to (\ref{1}) starting at \(u_0\). We also  set the corresponding Markov groups on bounded functionals
\[
\phi_N^t f(v) = f(\phi_N^t(P_N v)); \quad \quad C_b(L^2) \to C_b(L^2),
\]
and on probability measures
\[
\phi_N^{t*} \lambda(\Gamma) = \lambda(\phi_N^{-t}(\Gamma)); \quad\quad P(L^2) \to P(L^2).
\]
\end{remark}    

\subsection{Deterministic devices for the globalization}
Here we prove two results that play key role in the globalization procedure. The convergence lemma locally compares the NLS with its Galerkin projections. The globalization lemma is rather a general result that exploits uniform global well-posedness of Galerkin projections together with the convergence lemma in order to deduce GWP of the NLS. These two devices heavily rely on the LWP theory established previously and critically make use of the uniformity property of the time of increment $T$.
\begin{lemma}[Convergence lemma]\label{local uniform}
     Let $r,p, \delta$ with $r>\max(2,2q)$ be defined as in Proposition \ref{local well posedness}. Let $s>\frac{d}{p}+\delta$.~~Let $u_0\in H^s$ and  $(u_{0,N})_{N\in \mathbb{N}}$ a sequence such that~~$u_{0,N}\to u_0 ~~\text{in}~~ H^s$ with $u_{0,N}\in E_N$ .~~Let $R > 0,B_R := B_R(H^s)~~\text{and}~~T:=T(R,s)$ be the associated uniform existence time for the problem (\ref{1}) with initial data $u_{0,N}$; we have $$\forall \quad\frac{d}{p}+\delta<s'<s, \sup_{u_0\in B_R}\lVert \phi^tu_0-\phi^t_N(u_{0,N})\rVert_{Y^{s'}_T}\to 0 ~~~~\text{as}~~~~ N\to \infty$$
\end{lemma}
\begin{proof}
    Let us write down the Duhamel Formulas,
        \[
\phi^tu_0 = S(t)u_0 - i \int_0^t S(t-\tau) \left( |\phi^{\tau}u_0|^{2q} \phi^{\tau}u_0 \right) d\tau,
\] 
    \[
\phi^t_Nu_{0,N} = S(t)u_{0,N} - i \int_0^t S(t-\tau) P_N \left( |\phi^{\tau}_Nu_{0,N}|^{2q} \phi^{\tau}_N u_{0,N} \right) d\tau.
\] We take the difference
    \begin{align*} 
    \phi^tu_0-\phi^t_Nu_{0,N}=S(t)(u_0-u_{0,N})&- i \int_0^t S(t-\tau)\ P_N \left( |\phi^{\tau}u_0|^{2q} \phi^{\tau}u_0 -|\phi^{\tau}_N u_{0,N}|^{2q} \phi^{\tau}_Nu_{0,N} \right)d\tau \\
    &- i \int_0^t S(t-\tau) \left( |\phi^{\tau}u_0|^{2q} \phi^{\tau}u_0 -P_N(|\phi^{\tau}u_0|^{2q} \phi^{\tau} u_0)\right)d\tau.
    \end{align*}
    We have
    \begin{align*}
        \lVert \phi^tu_0-\phi^t_N u_{0,N} \rVert_{H^{s'}}\leq \lVert u_0-u_{0,N}\rVert_{H^{s'}}&+C_2\int_0^t\left(\lVert \phi^{\tau}u_0\rVert_{L^{\infty}}^{2q}+\lVert \phi^{\tau}_N u_{0,N}\rVert_{L^{\infty}}^{2q} \right)\left(\lVert \phi^{\tau}u_0-\phi^{\tau}_N u_{0,N}\rVert_{H^{s'}} \right)d\tau\\
        &+\int_0^t\lVert (1-P_N)(\abs{\phi^{\tau}u_0}^{2q}\phi^{\tau}u_0)\rVert_{H^{s'}}d\tau.
    \end{align*}
    Now, \[\lVert (1-P_N)f\rVert_{H^{s'}}\leq (1+\lambda_N)^{\frac{s'-s}{2}}\lVert(1-P_N)f\rVert_{H^s}\leq (1+\lambda_N)^{\frac{s'-s}{2}}\lVert f\rVert_{H^{s}}\] and $W^{\sigma,p} \hookrightarrow L^{\infty}$; so that we obtain
    \begin{align*}
    \lVert \phi^t u_0 - \phi^t_N u_{0,N} \rVert_{H^{s'}} 
    &\leq \lVert u_0-u_{0,N}\rVert_{H^{s}}+(1 + \lambda_N)^{\frac{s' - s}{2}} \int_0^t \lVert \phi^{\tau} u_0 \rVert_{W^{\sigma, q}}^{2q} 
    \lVert \phi^{\tau} u_0 \rVert_{H^s} \, d\tau  \\
    &\quad + C_2 \int_0^t \left( \lVert \phi^{\tau} u_0 \rVert_{W^{\sigma, p}}^{2q} 
    + \lVert \phi^{\tau}_N u_{0,N} \rVert_{W^{\sigma, p}}^{2q} \right) 
    \lVert \phi^{\tau} u_0 - \phi^{\tau}_N  u_{0,N} \rVert_{H^{s'}} \, d\tau.
    \end{align*}
  By applying the Gronwall lemma, we have 
  \begin{align*}
    &\lVert \phi^t u_0 - \phi^t_N u_{0,N} \rVert_{H^{s'}} \leq \left(\lVert u_0-u_{0,N}\rVert_{H^{s}}+  (1 + \lambda_N)^{\frac{s' - s}{2}}  
    \int_0^t \lVert \phi^{\tau} u_0 \rVert_{W^{\sigma, p}}^{2q} 
    \lVert \phi^{\tau} u_0 \rVert_{H^s}  \, d\tau \right) +\\
    &\int_0^t\left( \lVert u_0-u_{0,N} \rVert_{H^{s}} 
     +(1 + \lambda_N)^{\frac{s' - s}{2}}\int_0^{\tau} \lVert \phi^{\epsilon} u_0 \rVert_{W^{\sigma, q}}^{2q} 
    \lVert \phi^{\epsilon} u_0 \rVert_{H^s} \,  \right)C_2 \left( \lVert \phi^{\tau} u_0 \rVert_{W^{\sigma, q}}^{2q} 
    + \lVert \phi^{\tau}_N u_{0,N} \rVert_{W^{\sigma, q}}^{2q} \right)e^{\int_{\tau}^{t} f({\epsilon})d\epsilon}d\tau,
    \end{align*}
    where $f(\epsilon)=C_2\left(\lVert \phi^{\epsilon} u_0 \rVert_{W^{\sigma, q}}^{2q} 
    + \lVert \phi^{\epsilon}_N u_{0,N} \rVert_{W^{\sigma, p}}^{2q} \right)$.\\\\
    Now, let us remark that on $[0,T_R)$, we have $\lVert \phi^{\tau}u_0\rVert_{H^s}$;~~$\lVert \phi^{\tau}_Nu_{0,N}\rVert_{H^s} \leq C(R)$ and \\\\
    $\int_{\tau_1}^{\tau_2}\lVert \phi^{\tau} u_0 \rVert_{W^{\sigma, p}}^{2q}d\tau\leq T^{\gamma}\lVert\phi^{\tau}u_0\rVert^{2q}_{L^r_TW^{\sigma,p}}$;~~$\int_{\tau_1}^{\tau_2}\lVert \phi^{\tau}_N u_{0,N} \rVert_{W^{\sigma, p}}^{2q}d\tau\leq T^{\gamma}\lVert\phi^{\tau}_Nu_{0,N}\rVert^{2q}_{L^r_TW^{\sigma,p}} \leq C_1(R).$\\ \\
  Therefore, we arrive at  $$\sup_{u_0\in B_R}\lVert \phi^tu_0-\phi^t_N(u_{0,N})\rVert_{X^{s'}_T}\to 0 ~~~~\text{as}~~~~ N\to \infty.$$\\
    On the other hand, we have by using Strichartz estimate,
    \begin{align*}
        \lVert \phi^tu_0-\phi^t_Nu_{0,N} &\rVert_{L^r_TW^{\sigma',p}}\leq C_T\left(\lVert u_0-u_{0,N}\rVert_{H^{s'}}+\int_0^T\lVert (1-P_N)(\abs{\phi^{\tau}u_0}^{2q}\phi^{\tau}u_0)\rVert_{H^{s'}}d\tau\right)\\
        &+C_TC_2\int_0^T\left(\lVert \phi^{\tau}u_0\rVert_{L^{\infty}}^{2q}+\lVert \phi^{\tau}_N u_{0,N}\rVert_{L^{\infty}}^{2q} \right)\left(\lVert \phi^{\tau}u_0-\phi^{\tau}_Nu_{0,N}\rVert_{H^{s'}} \right)d\tau\\\\
        &\leq C_T\left(\lVert u_0-u_{0,N}\rVert_{H^{s}}+(1 + \lambda_N)^{\frac{s' - s}{2}}\left(C(R)\right )+C_3(R)\lVert \phi^t u_0-\phi^t_N(u_{0,N})\rVert_{X^{s'}_T}\right).
   \end{align*}
   So, $$\sup_{u_0\in B_R}\lVert \phi^tu_0-\phi^t_Nu_{0,N} \rVert_{L^r_TW^{\sigma',p}} \to 0 ~~\text{as}~N\to \infty.$$
   Overall, 
   $$\forall \quad\frac{d}{p}+\delta<s'<s,\sup_{u_0\in B_R}\lVert \phi^tu_0-\phi^t_N(u_{0,N})\rVert_{Y^{s'}_T}\to 0 ~~~~\text{as}~~~~ N\to \infty,$$
   which concludes the proof.
    \end{proof}

    \begin{lemma}[A sufficient condition for global existence in $H^s$]\label{lemma_sufficient_cond} Let $u_0 \in H^s$, Let $r,p, \delta$ with $r>\max(2,2q)$ be defined as in the above lemma. Let $s,s'$ such that $\frac{d}{p}+\delta<s'<s.$ Assume that $\phi^tu_0\in Y^{s'}_T$ for all $T>0$, then $\phi^tu_0$ is global (and belongs to $Y^s_T$ for all $T$). 
    \end{lemma} 
\begin{proof}
  From the Duhamel formula, we write
    \begin{align*}
\lVert \phi^tu_0 \rVert_{H^s} \leq \lVert u_0 \rVert_{H^s} + &C_1\int_0^t \lVert \phi^{\tau}u_0 \rVert_{L^{\infty}}^{2q} \lVert \phi^{\tau}u_0 \rVert_{H^s}  \, d\tau\\
 &\leq \lVert u_0 \rVert_{H^s} + C \int_0^t \lVert \phi^{\tau}u_0 \rVert^{2q}_{W^{\sigma',p}} \lVert \phi^{\tau}u_0 \rVert_{H^s}  \, d\tau.
   \end{align*}
According to the Gronwall Lemma, we have $$ \lVert\phi^t u_0 \rVert_{H^s}  \leq e^{C \int_0^t \| \phi^\tau u_0 \|^{2q}_{W^{\sigma',p}}\, d\tau} \lVert u_0\rVert_{H^s} .$$
Therefore, if for some initial datum $u_0 \in H^s$ we have that $  \int_0^T \| \phi^\tau u_0 \|^{2q}_{W^{\sigma',p}}\, d\tau<\infty$~~$\forall T>0,$ then the solution $\phi^tu_0$ will belong to $C([-T,T];H^s)$.\\
Next, by Holder inequality, we have
$$\int_0^T \| \phi^\tau u_0 \|^{2q}_{W^{\sigma',p}}\, d\tau\leq T^{\gamma}\left(\int_0^T \| \phi^\tau u_0 \|^{r}_{W^{\sigma',p}}\, d\tau\right)^{\frac{2q}{r}}\leq T^{\gamma}\lVert \phi^tu_0\rVert_{Y^{s'}_T}^{2q}.$$\\
So,
$$\lVert\phi^t u_0 \rVert_{H^s}  \leq e^{C T^{\gamma} \lVert \phi^tu_0\rVert_{Y^{s'}_T}^{2q}}\lVert u_0\rVert_{H^s}. $$
On the other hand, we have
\begin{align*}
\lVert \phi^{\tau}u_0 \rVert_{L^r_TW^{\sigma,p}} &\leq C(T)\left(\lVert u_0 \rVert_{H^s}  + \int_0^T \lVert \phi^{\tau}u_0 \rVert_{\infty}^{2q} \lVert \phi^{\tau}u_0 \rVert_{H^s}  \, d\tau\right)\\  
&\leq C(T)\left(\lVert u_0 \rVert_{H^s}  + \int_0^T \lVert \phi^{\tau}u_0 \rVert^{2q}_{W^{\sigma',p}} \lVert \phi^{\tau}u_0 \rVert_{H^s}  \, d\tau\right)\\
&\leq C(T)\left( \lVert u_0 \rVert_{H^s}+e^{CT^{\gamma}\lVert \phi^t u_0\rVert_{Y^{s'}_T}^{2q}}\lVert u_0 \rVert_{H^s}  ~T^{\gamma}\lVert \phi^t u_0\rVert_{Y^{s'}_T}^{2q}\right).
\end{align*}
Therefore we have $$\lVert \phi^t u_0\rVert_{Y^s_T}\leq e^{C T^{\gamma} \lVert \phi^tu_0\rVert_{Y^{s'}_T}^{2q}}\lVert u_0\rVert_{H^s}  +C(T)\left( \lVert u_0 \rVert_{H^s} +e^{CT^{\gamma}\lVert \phi^t u_0\rVert_{Y^{s'}_T}^{2q}}\lVert u_0 \rVert_{H^s}  ~T^{\gamma}\lVert \phi^t u_0\rVert_{Y^{s'}_T}^{2q}\right).$$
The proof is complete.
\end{proof}

\begin{lemma}[Globalization lemma]\label{globalization lemma}
  Let $r,p, \delta$ with $r>\max(2,2q)$ be defined as above. Let $s,s'$ such that $\frac{d}{p}+\delta<s'<s$. 
  \begin{itemize}
      \item Let $u_0\in H^s$ and $T_0$ be the associated existence time, let  $(u_{0,N})_{N\in \mathbb{N}}$ be a sequence such that $u_{0,N}\in E_N$ and ~~$u_{0,N}\to u_0 ~~\text{in}~~ H^s$
      \item Assume that there exists a non decreasing function $f:[0,+\infty)\to [0,+\infty) $ independent of N such that\\ $\lVert \phi^{t}_Nu_{0,N}\rVert_{H^{s'}} \leq f(|t|) ~~\forall t\in \mathbb{R}$.
  \end{itemize}
  Then
  \begin{enumerate}
  \item $\phi^tu_0$ is global in $H^{s'}$;
      \item $\lVert \phi^t u_{0}\rVert_{H^{s'}} \leq f(|t|) ~~\forall t\in \mathbb{R}$;
      \item $\phi^tu_0$ is global in $H^{s}$.
  \end{enumerate}
\end{lemma}
\begin{proof}
From the hypothesis, we have, for $|t|\leq T$,
\begin{align}
    \sup_{|t|\leq T}\|\phi^tu_Nu_{0,N}\|_{H^{s'}}\leq f(T).
\end{align}
Then, in particular,
\begin{align}
     \|u_{0,N}\|_{H^{s'}}\leq f(T)< \Lambda,
\end{align}
which implies that $(u_{0,N})_N\subset B_\Lambda$. Passing to the limit, we obtain that $u_0\in B_\Lambda$.

Let $T>>1$ be arbitrary. It suffices to show that the solution $\phi^tu_0$ exists on $[-T,T]$. Set $\Lambda=f(T)+1.$ Consider the ball $B_\Lambda(H^{s'})$ and the associated time existence $T_0.$

We write, with the use of estimation arguments similar to lemmas above,
\begin{align*}
\phi^t(u_0) - \phi_N^t(u_{0,N}) &= S(t)(u_0 - u_{0,N}) \quad -i \int_0^t S(t-\tau) P_N \left( |\phi^\tau(u_0)|^{2q} \phi^\tau(u_0) - |\phi_N^\tau(u_{0,N})|^{2q} \phi_N^\tau(u_{0,N}) \right) d\tau \\
&\quad -i \int_0^t S(t-\tau)(1 - P_N) |\phi^\tau(u_0)|^{2q} \phi^\tau(u_0) \, d\tau;
\end{align*}
\begin{align*}
\lVert \phi^t u_0 - \phi^t_N u_{0,N} \rVert_{Y_{T_0}^{s'}} 
&\lesssim \lVert u_0 - u_{0,N} \rVert_{H^{s'}} \nonumber 
+  \int_{0}^{T_0} \left( \lVert \phi^{\tau} u_0 \rVert_{W^{\sigma,p}}^{2q} 
+ \lVert \phi^{\tau} u_{0,N} \rVert_{W^{\sigma,p}}^{2q} \right) 
\lVert \phi^{\tau} u_0 - \phi^{\tau} u_{0,N} \rVert_{H^{s'}} \, d\tau  \\
&\quad + (1 + \lambda_N)^{\frac{s' - s}{2}} \int_0^{T_0} 
\lVert \phi^{\tau} u_0 \rVert_{W^{\sigma, p}}^{2q} 
\lVert \phi^{\tau} u_0 \rVert_{H^s} \, d\tau.
\end{align*}
Using the properties of $T_0$, we infer that
\begin{align}
 \| \phi^t u_0 - \phi^t_N u_{0,N} \|_{Y_{T_0}^{s'}}\lesssim \|u_0-u_{0,N}\|_{H^{s'}}+(1+\lambda_N)^{\frac{s'-s}{2}}.   
\end{align}
In particular, as $N\to\infty$,
$$\| \phi^t(u_0) - \phi_N^t (u_{0,N}) \|_{X^{s'}_{T_0}} \to 0.$$

  We have by using the triangle inequality and the passage to the limit that $\forall|t|\leq T_0$,  
  $$\lVert \phi^t u_0\rVert_{H^{s'}}\leq \lVert \phi^t u_0-\phi_N^tu_{0,N}\rVert_{H^{s'}}+\lVert \phi_N^t u_{0,N}\rVert_{H^{s'}}\leq o(1)+\Lambda.$$
  We obtain that $\phi^{T_0}u_0$ still belongs to $B_\Lambda$. \\
By repeating this procedure $n$ times, with $nT_0\leq T$, we will see that \, $\phi^{nT_0}(u_0)$\, remains in the ball $B_{\Lambda}$, which allows to extend the solution on $[-T,T]$, as desired.
   Moreover, as $N\to \infty$ $$\forall t\in [-T,T],~~~ \lVert \phi^t u_0\rVert_{H^{s'}}\leq \lVert \phi^t u_0-\phi^tu_{0,N}\rVert_{H^{s'}}+\lVert \phi^t_N u_{0,N}\rVert_{H^{s'}}\leq o(1)+ f(|t|).$$ 
  Now, Lemma \ref{lemma_sufficient_cond} implies that  $\phi^tu_0$ is global in $H^{s}.$ 
\end{proof}

    \leavevmode\par
\section{Dissipation models and apriori estimates}\label{Sect. Dissip_apriori}
In Section \ref{Sect_Glob}, we aim use the IID limit method the globalize the local solutions constructed previously. A key step in this method lies in designing a suitable  dissipation model, which constitutes one of its challenging aspects. This is because it involves identifying dissipation that incorporates linear and nonlinear effects taking into account parameters related to dimensions and nonlinearity powers. Finding a dissipation mechanism that considers both linear and nonlinear phenomena can be a difficult task for supercritical regimes. In this section, we construct a dissipation model that captures all powers of the nonlinearity and dimensions. In some regimes, we will need the following very useful inequalities.
\subsection{Some useful inequalities}
\begin{proposition}\label{Prop-Est-Nonlinearity}
Let $\Delta$ be the Laplace-Beltrami operator on the Riemannian manifold ${M}$. Let $\gamma\in(0,1]$. For any $f\in C^\infty({M},\C)$, 
\begin{itemize}
\item we have the inequality
\begin{align}
\langle|f|^2f,(-\Delta)^\gamma f\rangle\geq \frac{1}{2}\|(-\Delta)^{\frac{\gamma}{2}}|f|^2\|_{L^2}^2. 
\end{align}
  \item for any real number $q>1$, we have the inequality
  \begin{align}
\langle|f|^{2q}f,(-\Delta)^\gamma f\rangle\geq \frac{1}{q+1}\|(-\Delta)^{\frac{\gamma}{2}}|f|^{q+1}\|_{L^2}^2. 
\end{align}
\end{itemize}
 \end{proposition}

 The proof
 of the proposition relies on the following famous inequality of C\'ordoba-C\'ordoba \cite{cordoba2003pointwise}.
\begin{lemma}[C\'ordoba-C\'ordoba inequality]\label{a}
Under the setting of the proposition above, let $\Phi$ be a convex $C^2(\mathbb{R},\mathbb{R})$-function satisfying $\Phi(0)=0$, and $\gamma \in (0,1]$. For any $f \in  C^\infty({M}, \mathbb{R})$, the inequality
\begin{align}
    \Phi'(f)(-\Delta)^{\gamma}f\geq (-\Delta)^{\gamma}\Phi(f)
\end{align}
holds pointwise almost everywhere in ${M}$.
\end{lemma}
\begin{proof}[Proof of Proposition \ref{Prop-Est-Nonlinearity}]
    Let us start by proving a complex version of the C\'ordoba-C\'ordoba inequality. For $f\in C^\infty({M},\C)$
    \begin{align}
        \mathcal{R}[f(-\Delta)^{\gamma} \bar{f}]\geq \frac{1}{2}(-\Delta)^{\gamma}|f|^2.
    \end{align}
    Indeed, let us write $f=a+ib$. We have that
    \begin{align}
        \mathcal{R}[f(-\Delta)^{\gamma} \bar{f}]=a(-\Delta)^{\gamma}a+b(-\Delta)^{\gamma}b\geq \frac{1}{2}(-\Delta)^{\gamma}(a^2+b^2)=\frac{1}{2}(-\Delta)^{\gamma}|f|^2,
    \end{align}
    where, in the last step, we used the (real) C\'ordoba-C\'ordoba inequality.\\
    Now, applying the complex inequality above, we have
    \begin{align}
       \langle|f|^{2q}f,(-\Delta)^\gamma f\rangle\geq \frac{1}{2}\langle |f|^{2q},(-\Delta)^\gamma |f|^2 \rangle.
    \end{align}
    We distinguish two cases:
    \begin{itemize}
        \item for $q=1$. We obtain the corresponding claim.
        \item for $q>1$, we write
        \begin{align}
       \langle|f|^{2q}f,(-\Delta)^\gamma f\rangle &\geq \frac{1}{2}\langle |f|^{2q},(-\Delta)^\gamma |f|^2 \rangle\\
       &=\frac{1}{2}\langle (|f|^2)^{\frac{q+1}{2}},(|f|^2)^{\frac{q-1}{2}}(-\Delta)^\gamma |f|^2 \rangle
    \end{align}
    and apply the real C\'ordoba-C\'ordoba inequality with $\Phi(x)=x^{\frac{q+1}{2}}$ (which is convex for $q>1$), and obtain 
    \begin{align}
        \frac{1}{2}\langle (|f|^2)^{\frac{q+1}{2}},(|f|^2)^{\frac{q-1}{2}}(-\Delta)^\gamma |f|^2 \rangle\geq \frac{1}{2}\frac{2}{q+1}\langle (|f|^2)^{\frac{q+1}{2}},(-\Delta)^\gamma (|f|^2)^{\frac{q+1}{2}} \rangle.
    \end{align}
    \end{itemize}
      Hence we obtain the claim.
\end{proof}

\subsection{The dissipation model and estimation of the dissipation rates}
 Let $\epsilon>0$ be small enough, introduce the dissipation operator 
 \begin{align}
        \mathcal{L}_s(u)=(-\Delta)^{s-1}u+C_{d,s}\lVert u \rVert_{H^{s^-}}^{3\epsilon^{-1}}u 
   \end{align}
    where $s_{M^d}<s$, $C_{d,s}$ and $\epsilon$ will be determined below (as required of the estimation of the dissipation rate of energy). \\ Here, we derive some apriori estimates on the dissipation rates of the mass $M(u)$ and the energy $E(u)$ of \eqref{Intro_NLS}. We denote $\tilde{k}=\epsilon^{-1}$ and use both notations.

The dissipation rate of a functional $F(u)$ against $\mathcal{L}_s(u)$ is given by
\begin{align}
    \mathcal{F}(u)=\langle \delta_u F,\mathcal{L}_s(u)\rangle,
\end{align}
where $\delta_u F$ stands for the functional derivative (in the sense of Gateaux) at $u$ of a functional $F$ on $L^2$:
\begin{align}
    F'(u,v)=\langle\delta_uF,v\rangle\quad \forall v\in L^2.
\end{align}
Hence, for the mass we obtain
\begin{align*}
\mathcal{M}(u)&= \langle \delta_u M, \mathcal{L}_s(u)\rangle= \lVert u\rVert_{H^{s-1}}^2+C_{d,s}\lVert u\rVert_{H^{s^-}}^{3\tilde{k}}\lVert u\rVert_{L^2}^2.
\end{align*}
As for the dissipation rate of the energy, the estimations are more involved:
\begin{align}
     \mathcal{E}(u)&=\langle \delta_u E, \mathcal{L}_s(u)\rangle=\lVert u\rVert_{H^s}^2+C_{d,s}\lVert u\rVert_{H^{s^-}}^{3\tilde{k}}\lVert u\rVert_{L^{2q+2}}^{2q+2}+C_{d,s}\lVert u\rVert_{H^{s^-}}^{3\tilde{k}}\lVert u\rVert_{H^1}^2+\langle (-\Delta)^{s-1}u,|u|^{2q}u\rangle. 
\end{align}
In order to show that the quantity above is coercive, we need to absorb the last term by the first ones. We separate two cases:
\begin{itemize}
    \item For $s_{M^d}<s\leq 2$, we have by using Proposition \ref{Prop-Est-Nonlinearity}
$$ \mathcal{E}(u)\geq \lVert u\rVert_{H^s}^2+C_{d,s}\lVert u\rVert_{H^{s^-}}^{3\tilde{k}}\lVert u\rVert_{L^{2q+2}}^{2q+2}+C_{d,s}\lVert u\rVert_{H^{s^-}}^{3\tilde{k}}\lVert u\rVert_{H^1}^2+\frac{1}{q+1}\lVert (-\Delta)^{\frac{s-1}{2}}|u|^{q+1}\lVert^2_{\mathbb{L}^2}. $$
\item Let now $\frac{d}{2}\geq s>2$. We first notice that $d$ must be larger than $4$. Write 
$$\langle (-\Delta)^{s-1}u,|u|^{2q}u\rangle=\langle (-\Delta)^{\frac{s}{2}}u,(-\Delta)^{\frac{s-2}{2}}|u|^{2q}u\rangle\leq \frac{1}{2}\lVert u\rVert_{H^s}^2+\frac{1}{2}\lVert (-\Delta)^{\frac{s-2}{2}}|u|^{2q}u\rVert_{L^2}^2. $$\\
Now, by using the following chain (see e.g. Christ-Weinstein \cite{christ1991dispersion}) with $\frac{1}{2}=\frac{1}{2+\gamma}+\frac{1}{(2+\gamma)^*}$,
we have 
\begin{align*}
    \lVert (-\Delta)^{\frac{s-2}{2}}|u|^{2q}u\rVert_{L^2}^2 &\lesssim\lVert|u|^{2q}\rVert_{L^{(2+\gamma)^*}}^2\lVert u\rVert_{W^{s-2,2+\gamma}}^2\\
    &\lesssim \lVert u\rVert_{L^{2q(2+\gamma)^*}}^{4q}\lVert u\rVert_{W^{s-2,2+\gamma}}^2.
\end{align*}
Now, we have $W^{s-1+\delta,2}\hookrightarrow W^{s-2,2+\gamma}$  with $\delta\in (0,1)$ 
 for ${1+\delta}=(\frac{d}{2}-\frac{d}{2+\gamma})^+$.
So, $2(1+\delta^-)(2+\gamma)=\gamma d$, then $\gamma=\frac{4(1+\delta^-)}{d-2-2\delta^-}$ and $(2+\gamma)^*=\frac{d}{1+\delta^-}.$ Recall that $d>4$, therefore $\gamma$ is still positive.\\\\
We remark that $H^{s-}\hookrightarrow L^{2q(2+\gamma)^*}$ if $s>\frac{d}{2}-\frac{d}{2q(2+\gamma)^*}=\frac{d}{2}-\frac{d}{\frac{2qd}{1+\delta^-}}=\frac{d}{2}-\frac{(1+\delta^-)}{2q}$ with $\delta\in(0,1).$\\
By taking $\delta=1^-$, we can then have this embedding for $s>\frac{d}{2}-\frac{1}{q}=s_{q,d}$.\\ \\
Now, let $s>\max(2,s_{q,d})$, $\delta=1^-$, so $s-1+\delta \approx s^-$, then $s-1+\delta=s^-(1-\beta)+\beta\times 0 $ with $\beta \in (0,1)$ is close enough to $0$.\\
 By interpolating $H^{s-1+\delta}$ between $H^{s^-}$ and $L^2$, we obtain
 \begin{align*}
 &\lVert (-\Delta)^{\frac{s-2}{2}}|u|^{2q}u\rVert_{L^2}^2 \lesssim \lVert u\rVert_{H^{s^-}}^{4q}\lVert u\rVert_{H^{s-1+\delta}}^2\\
 & \lesssim \lVert u\rVert_{H^{s^-}}^{4q+2-2\beta} \lVert u\rVert_{L^2}^{2\beta}      \\
 &\lesssim C+\lVert u\rVert_{L^2}^2\lVert u\rVert_{H^{s^-}}^{3\tilde{k},} ~~\text{according the Young's inequality where}\\
 &~~\quad \quad \quad \quad 3\tilde{k}=\frac{4q+2-2\beta}{\beta} ~~~~~\text{i.e}~\tilde{k}\to \infty ~\text{ large enough} (or\, \epsilon\to 0 \text{ small enough})\\
 &= 2K_{s,d}+ C_{d,s}\lVert u\rVert_{L^2}^2\lVert u\rVert_{H^{s^-}}^{3\tilde{k}}.
\end{align*}
Hence $$ \mathcal{E}(u)\geq \frac{1}{2}\lVert u\rVert_{H^s}^2+C_{d,s}\lVert u\rVert_{H^{s^-}}^{3\tilde{k}}\lVert u\rVert_{L^{2q+2}}^{2q+2}+ \frac{C_{d,s}}{2}\lVert u\rVert_{L^2}^2\lVert u\rVert_{H^{s^-}}^{3\tilde{k}}-K_{s,d} =\mathcal{E}_0(u)-K_{s,d}.$$
\end{itemize}
Overall, we have that
\begin{align}\label{Coercive_E}
    \mathcal{E}_0(u):=\frac{1}{2}\lVert u\rVert_{H^s}^2+C_{d,s}\lVert u\rVert_{H^{s^-}}^{3\tilde{k}}\lVert u\rVert_{L^{2q+2}}^{2q+2}+ \frac{C_{d,s}}{2}\lVert u\rVert_{L^2}^2\lVert u\rVert_{H^{s^-}}^{3\tilde{k}}\leq \mathcal{E}(u)+K_{s,d}.
\end{align}
Therefore an estimate from above on $\mathcal{E}(u)$ yields an estimate on $\mathcal{E}_0(u)$, which is a coercive functional. 

\leavevmode

\section{IID-limit and globalization}\label{Sect_Glob}
In this section, we perform our globalization argument following the IID limit strategy. Here are the main steps:
\begin{itemize}
\item Step 1: We consider a stochastic equation consisting in damped-driven Galerkin projection of NLS. We prove global well-posedness and construct stationary measures. Useful bounds on the measure are proven.
\item Step 2: We pass to the inviscid limit, hence recovering a statistical theory of the Galerkin projections of NLS.
\item Step 3: We construct ensembles selecting ``well behaved'' approximate solutions that enjoy uniform individual bounds allowing to use our previous globalization lemma.
\item Step 4: We prove properties of the statistical ensemble and of the measure.\\\\
\end{itemize}

We denote by $(\Omega, \mathcal{F},\mathbb{P})$ a complete probability space. For a Banach space $E$, we consider random variables $X:\Omega \to E$~~being Bochner measurable functions with respect to $\mathcal{F}$ and $\mathcal{B}(E)$, where $\mathcal{B}(E)$ is Borel $\sigma$-algebra of $E$.\\
 For any positive integer $N$, we define the $N$-dimension Brownian motion by \[\mathcal{W}^N(t,x)=\sum_{n=1}^Na_ne_n(x)\mathcal{B}_n(t)\] where $(\mathcal{B}_n(t)$ is a one dimensional independent Brownian motions with filtration $(\mathcal{F}_t)_{t\geq 0}$.\\
 The numbers $(a_n)_{n\geq 1}$ are such that $A^0=\sum_{n\geq 1}a_n^2<\infty$. \\
 We denote by $A_N^s=\sum_{n=1}^N \lambda_n^sa_n^2$ and $A^s=\sum_{n\geq1} \lambda_n^sa_n^2$ where $(\lambda_n)_{n\geq 1}$ are the eigenvalues of $-\Delta_g$ such that $A^{\frac{d}{2}-\frac{1}{2}}<\infty$.\\
 Let us recall that $(,)$ denote the real dot product in $L^2(M;\C)$.\\
 Let $F: E_N \to \mathbb{R}, u\mapsto F(u)$ be a smooth function. We denote by $F'(u;v)$ and $F''(u;v,w)$ respectively the first and second derivative of $F$ at $u$.\\


\subsection{Estimates on the stationary measures of the stochastic problems }
\leavevmode\\
Let us consider the following problem
\begin{equation}\label{eq3}
\left\{
\begin{aligned}
du &= \left[i \left(\Delta u - P_N ( |u|^{2q}u)\right)-\alpha\left( (-\Delta)^{s-1}u+C_{d,s}\lVert u \rVert_{H^{s^-}}^{3\tilde{k}}u \right)\right]dt+\sqrt{\alpha} \sum_{n=1}^N a_ne_n(x)d\mathcal{B}_n(t),\\
u(0) &= P_N u_0 \in E_N.
\end{aligned}
\right.
 \end{equation}
  \begin{definition}(Stochastic global well-posedness). 
  \leavevmode\\
We say that \eqref{eq3} is stochastically globally well-posed on $E_N$ if the following properties hold: \\ 
\begin{itemize}
    \item For any $E_N$-valued random variable $u_0$ independent of $\mathcal{F}_t$, 
    \begin{itemize}
        \item for $\mathbb{P}$-a.e. $\omega$, there exists a solution $u = u^\omega \in C(\mathbb{R}^+; E_N)$ of \eqref{eq3} with initial datum $u_0 = u_0^\omega$ in the integral sense, i.e.,
        \begin{equation*}
            u(t) = u_0 + \int_0^t \Big[i \big(\Delta u - P_N (|u|^{2q}u) \big) -\alpha\left( (-\Delta)^{s-1}u+C_{d,s}\lVert u \rVert_{H^{s^-}}^{3\tilde{k}}u\right)\Big] \, d\tau + \sqrt{\alpha} \mathcal{W}^N(t)
        \end{equation*}
        with equality as elements of $C(\mathbb{R}^+; E_N)$.
    \end{itemize}
    \item For any $u_1^\omega, u_2^\omega \in C(\mathbb{R}^+; E_N)$ two solutions with the same initial data $u_{1,0}^\omega = u_{2,0}^\omega$, then $u_1^\omega \equiv u_2^\omega$.\\
    \item The process $(\omega, t) \mapsto u^\omega(t)$ is adapted to the filtration $\sigma(u_0, \mathcal{F}_t)$.

\end{itemize}
\end{definition}
Using an argument of Da-Prato-Debussche \cite{da2002two}, we split the unknown as follows $u=v+z$, where
\begin{equation} \label{linear-stoch-eq}
\left\{
\begin{aligned}
dz_{\alpha} &= (i\Delta z_{\alpha}-\alpha(-\Delta)^{s-1}z_{\alpha})dt+\sqrt{\alpha}\sum_{n=1}^N a_ne_n(x)d\mathcal{B}_n(t)\\
z_{\alpha}(0)&=0
\end{aligned}
\right.
\end{equation}

\begin{equation}
\left\{
\begin{aligned}
\partial_t v&=i(\Delta v-P_N(|v+z_{\alpha}|^{2q}(v+z_{\alpha})))-\alpha\left((-\Delta)^{s-1}v+ \lVert v+z_{\alpha}\rVert_{H^{s-}}^{3\tilde{k}}(v+z_{\alpha})\right)\\
v(0)&=P_N u_0
\end{aligned}
\right.
\end{equation}
The equation (\ref{linear-stoch-eq}) is a linear stochastic equation whose solution is given by the stochastic convolution $$z_{\alpha}(t)=\sqrt{\alpha}\int_0^t e^{(t-\tau)(i\Delta-\alpha(-\Delta)^{s-1})}d\mathcal{W}^N.$$
We have by applying the Ito's formula
$$\lVert z_{\alpha}(t)\rVert^{2}_{L^2}=2\sqrt{\alpha}\int_0^t \sum_{n=1}^N a_n(z_{\alpha},e_n)d\mathcal{B}_n(\tau)-2\alpha\sqrt{\alpha}\int_0^t \lVert z_{\alpha}(\tau)\rVert_{H^{s-1}}^2d\tau+\sqrt{\alpha}\sum_{n=1}^N a_n^2t,$$
 $$\mathbb{E}\left(\lVert z_{\alpha}(t)\rVert^2_{L^2}+2\alpha\sqrt{\alpha}\int_0^t \lVert z_{\alpha}(\tau)\rVert_{H^{s-1}}^2d\tau\right)=\sqrt{\alpha}\sum_{n=1}^N a_n^2t.$$\\
 Let us fix $T>0$, 
 we have $$\mathbb{E}\left(2\alpha\sqrt{\alpha}\int_0^T \lVert z_{\alpha}(\tau)\rVert_{H^{s-1}}^2d\tau \right)\leq \sqrt{\alpha}\sum_{n=1}^N a_n^2T$$
 we have also by applying Doob's inequality, \\
 $$\mathbb{E}\left(\sup_{t\in [0,T]}\lVert {z_{\alpha}(t)}\rVert_{L^2} \right)^2\leq 4\mathbb{E}\left(\lVert {z_{\alpha}(T)}\rVert_{L^2}^2\right)\leq 4\sqrt{\alpha}\sum_{n=1}^N a_n^2T.$$\\
Consequently, for $\mathbb{P}$-almost all $\omega\in \Omega$, for all $T>0$, $$\sup_{t\in [0,T]}\lVert{z_{\alpha}(\omega,t)}\rVert^2_{L^2}\leq C_{\alpha}(w,T).$$\\
 We see clearly that  for all $\omega\in \Omega$, $z^{\omega}_{\alpha}\in C_t(\mathbb{R},C^{\infty}(E_N))$.\\\\
 Let consider the second equation

\[
\partial_t v=F(v)\quad
v(0)=P_N u_0
\]
where
\begin{align*}
  F:E_N &\to E_N\\
    v &\mapsto F(v)=i(\Delta v-P_N(|v+z_{\alpha}^\omega|^{2q}(v+z_{\alpha}^\omega)))-\alpha\left((-\Delta)^{s-1}v+ \lVert v+z_{\alpha}^\omega\rVert_{H^{s-}}^{3\tilde{k}}(v+z_{\alpha}^\omega)\right).  
\end{align*}
It is easy to see that $F\in C^{\infty}(E_N,E_N)$, then thanks to the Cauchy Lipschitz's theorem, the problem has a smooth local solution.
Now let us show that this solution is global in time $\mathbb{P}$-almost surely. We rely on the equivalence of norms in finite dimension.
\begin{proposition}
    The local solution $v$ constructed above exists globally in time, $\mathbb{P}$-almost surely.
\end{proposition}
\begin{proof}
    We calculate
    \begin{align*}
        \frac{d}{dt}\left(\frac{\lVert{v}\rVert_{L^2}^2}{2}\right)&=\left(i(\Delta v-P_N(|v+z_{\alpha}^\omega|^{2q}(v+z_{\alpha}^\omega)))-\alpha\left((-\Delta)^{s-1}v+ \lVert v+z_{\alpha}^\omega\rVert_{H^{s-}}^{3\tilde{k}}(v+z_{\alpha}^\omega)\right),v\right)\\
        &\leq \frac{\lVert{z_{\alpha}^\omega}\rVert_{L^2}^{1-\frac{1}{\beta}}}{C(\beta)}+C_1(\beta)\lVert v+z_{\alpha}^\omega\rVert_{H^{s-}}^{3\tilde{k}}-\alpha \lVert v\rVert_{H^{s-1}}^2-\alpha \lVert v+z_{\alpha}^\omega\rVert_{H^{s-}}^{3\tilde{k}}\lVert v \rVert_{L^2}^2\\
        &\quad \quad+\alpha \lVert v+z_{\alpha}^\omega\rVert_{H^{s-}}^{3\tilde{k}}\left(\frac{\lVert v\rVert_{L^2}^2}{2}+\frac{\lVert z_{\alpha}^\omega\rVert_{L^2}^2}{2}\right)~\text{with}~~\beta=\frac{3\tilde{k}}{2q+1}
        \\
        & \leq \frac{\lVert{z_{\alpha}^\omega}\rVert_{L^2}^{1-\frac{1}{\beta}}}{C(\beta)}+\frac{\alpha}{2}\lVert v+z_{\alpha}^\omega\rVert_{H^{s-}}^{3\tilde{k}}\left(C_2(\beta,\alpha)+\lVert z_{\alpha}^\omega\rVert_{L^2}^2-\lVert v\rVert_{L^2}^2\right).   
    \end{align*}
    \\
     We have that for $\mathbb{P}$-almost all $w\in \Omega$, for all $T$, , there is constant $C_{\alpha}(w,T)$ such that~~ $$\sup_{t\in [0,T]}\lVert z_{\alpha}(w,t)\rVert^2_{L^2}\leq C_{\alpha}(w,T).$$\\
    For a fixed $\omega$, let $t\in [0,T]$. We have the alternative:\\
    \begin{itemize}
        \item Either $\lVert v \rVert^2_{L^2}\leq C(\beta,\alpha)+C_{\alpha}(w,T),$\\
        \item or $\lVert v \rVert^2_{L^2}>C(\beta,\alpha)+C_{\alpha}(w,T). $
    \end{itemize}
    In the second case, we will have~~$ \lVert v+z_{\alpha}^w\rVert_{H^{s-}}^{3\tilde{k}}\left(C_2(\beta,\alpha)+\lVert z_{\alpha}^w\rVert_{L^2}^2-\lVert v\rVert_{L^2}^2\right)<0 $.
    So, in any case $$\frac{d}{dt}\left(\frac{\lVert{v}\rVert_{L^2}^2}{2}\right)\leq M_{\alpha,\omega,T,\beta}+\frac{\lVert{z_{\alpha}^w}\rVert_{L^2}^{1-\frac{1}{\beta}}}{C(\beta)}\leq C^1_{\alpha,\beta}(w,T).$$\\
    Hence $$\sup_{t\in[0,T]}\lVert{v}\rVert_{L^2}^2\leq \lVert{u_0}\rVert_{L^2}^2+C^2_{\alpha,\beta}(w,T).$$
    This is the needed control to perform an iteration of the local solution.
\end{proof}
\subsection*{Uniqueness and continuity} By using the fact that $F\in C^{\infty}(E_N,E_N)$ and the mean value theorem, it is easy to prove the uniqueness and continuity.\\
\subsection*{Adaptability}It is clear that the solution $z_{\alpha}(t)$ is adapted to $\mathcal{F}_t$ and since $v$ is constructed by using the fix point theorem, then $v$ also is adapted, therefore $u$ is adapted.\\
Let us denote by $u_{\alpha}(t,u_0)$ the unique stochastic solution to (\ref{eq3})

 \leavevmode
\subsection{The Markov semi-group and stationnary measure}
Let us define the transition probability $P_t^{\alpha,N}(u_0,\Gamma)=\mathbb{P}(u_{\alpha}(t,P_Nu_0)\in \Gamma)$ with $u_0\in L^2$ and $\Gamma \in \text{Bor}(L^2)$ .\\
The Markov semi-groups are given by:
\begin{align*}
    \mathcal{B}_t^{\alpha,N} f(u_0) &= \int_{L^2} f(v) P_t^{\alpha,N}(u_0, dv) \quad \quad C_b(L^2)\to C_b(L^2)\\
    \mathcal{B}_t^{\alpha,N *} \lambda(\Gamma) &= \int_{L^2} \lambda(du_0) P_t^{\alpha,N}(u_0, \Gamma) \quad \quad P(L^2) \to P(L^2).
\end{align*}
From the continuity of the solution with respect to initial data, we obtain the Feller property:
\[
\mathcal{B}_{t}^{\alpha,N} C_b(L^2) \subset C_b(L^2).
\]
\begin{proposition}
    Let $u_0$ be a random variable in $E_N$ independent of $\mathcal{F}_t$ such that $\mathbb{E}M(u_0)<\infty$ and $\mathbb{E}E(u_0)<\infty$. Let u be the solution to (\ref{eq3}) starting at $u_0$. Then we have
    \begin{align}\label{estimate1}
      \mathbb{E}M(u)+\alpha \int_0^t\mathbb{E}\mathcal{M}(u)d\tau &= \mathbb{E}M(u_0)+\alpha\frac{A^0_N}{2}t;\\ \label{estimate2}
      \mathbb{E}E(u)+\alpha \int_0^t\mathbb{E}\mathcal{E}(u)d\tau &\leq \mathbb{E}E(u_0)+\frac{\alpha}{2}\left(A^1_Nt+A^{\frac{d}{2}-\frac{1}{2}}_N\int_0^t\mathbb{E}\lVert{u}\rVert_{L^{2q}}^{2q}d\tau\right).
    \end{align}
\end{proposition}
\begin{proof}
    Let $F(u)$ be $M(u)$ or $E(u)$.
    We have by Ito's Formula,
\begin{align*}
         dF&=\langle \nabla_{{u}}F, du\rangle+\frac{\alpha}{2}\sum_{n=1}^Na^2_n\langle \nabla^2_uF;e_n,e_n\rangle dt\\
         &=\langle \nabla_{{u}}F, i(\Delta u -P_N(\abs{u}^{2q}u))\rangle-\alpha\langle \nabla_{{u}}F, (-\Delta)^{s-1}u+C_{d,s}\lVert u \rVert_{H^{s^-}}^{3\tilde{k}}u \rangle dt \\
         &~~~~~~~~~~\quad \quad+\sqrt{\alpha}\sum_{n=1}^Na_n\langle\nabla_{{u}}F, e_n\rangle d\mathcal{B}_n(t)
         +\frac{\alpha}{2}\sum_{n=1}^Na^2_n\langle\nabla_{u}^2F;e_n^2\rangle.
     \end{align*}
     For $F=M(u)$, we have
     $$M(u)=M(u_0)-\alpha\int_0^t\langle u,(-\Delta)^{s-1}u+C_{d,s}\lVert u \rVert_{H^{s^-}}^{3\tilde{k}}u \rangle d\tau +\sqrt{\alpha}\sum_{n=1}^Na_n\int_0^t\langle u,e_n\rangle d\mathcal{B}_n(\tau) +\frac{\alpha}{2}\sum_{n=1}^Na^2_nt.$$\\
     By taking the expectation and using the vanishing property of the martingale term, we obtain $$ \mathbb{E}M(u)+\alpha \int_0^t\mathbb{E}\mathcal{M}(u)d\tau = \mathbb{E}M(u_0)+\alpha\frac{A^0_N}{2}t. $$
     For \( F = E(u) \), we have
\begin{align*}
{dE}(u) \leq  \langle -\Delta u&+\abs{u}^{2q}u ~,~ i(\Delta u -P_N(\abs{u}^{2q}u))\rangle dt-\alpha \langle -\Delta u+\abs{u}^{2q}u,(-\Delta)^{s-1}u+C_{d,s}\lVert u \rVert_{H^{s^-}}^{3\tilde{k}}u \rangle dt \\
&+\sqrt{\alpha}\sum_{n=1}^Na_n\langle -\Delta u+\abs{u}^{2q}u, e_n\rangle d\mathcal{B}_n(t)
+\frac{\alpha}{2}\sum_{n=1}^Na^2_n(\langle -\Delta e_n,e_n\rangle+\langle \abs{u}^{2q};e_n,e_n\rangle)dt
\end{align*}
Thus, 
\[
{E}(u) \leq {E}(u_0) - \alpha \int_0^t \mathcal{E}(u) d\tau + \sqrt{\alpha}\int_0^t\sum_{n=1}^Na_n\langle -\Delta u+\abs{u}^{2q}u, e_n\rangle d\mathcal{B}_n(\tau)+\frac{\alpha}{2}\left(A^1_Nt+\sum_{n=1}^Na^2_n\lVert{e_n}\rVert^2_{L^{\infty}}\int_0^t\lVert{u}\rVert^{2q}_{L^{2q}}d\tau \right).
\]
Therefore, after using Sogge's inequalities \cite{sogge1988concerning}, we have
\[
\mathbb{E}E(u)+\alpha \int_0^t\mathbb{E}\mathcal{E}(u)d\tau \leq \mathbb{E}E(u_0)+\frac{\alpha}{2}\left(A^1_Nt+A^{\frac{d}{2}-\frac{1}{2}}_N\int_0^t\mathbb{E}\lVert{u}\rVert_{L^{2q}}^{2q}d\tau\right).
\] Hence the claim.
\end{proof}
\leavevmode
\subsection*{Existence of Stationnary measure}
\leavevmode
\begin{theorem}

For any \( N \geq 2 \) and \( \alpha \in (0,1) \), there is a stationary measure $\mu_N^{\alpha}$
and we have the following estimates:
\begin{align}\label{estimate3}
\int_{L^2} \mathcal{M}(u) \mu_{N}^{\alpha}(du) &=\frac{A_N^0}{2}\leq \frac{A^0}{2}\\ \label{estimate4}
\int_{L^2} \mathcal{E}(u) \mu_{N}^{\alpha}(du) &\leq C, 
\end{align}
where $C$ does not depend on $\alpha$ and $N$.

\end{theorem}
\begin{proof}
Let $B_R$ be the ball of $H^{s-1}$ with center $0$ and raduis $R.$ By using the Bogoliubov-Krylov argument, we have to find the measure $\lambda$ such that the sequence $\{\Bar{\lambda_t}\}_{t\geq 0}$ is tight, where $$\Bar{\lambda_t}=\frac{1}{t}\int_0^t\mathcal{B}_{\tau}^{\alpha,N,*}\lambda ~d\tau.$$
By Prokhorov theorem, Let $R>0$, we just need to find a compact $K$ of $L^2$ such that $\bar{\lambda}_t(K^c)<\epsilon(R)\to 0$ as $R\to\infty$.\\
We have 
\begin{align*}
    \bar{\lambda}_t(B_R^c)&=\frac{1}{t}\int_0^t\mathcal{B}_{\tau}^{\alpha,N,*}\lambda(B_R^c) ~d\tau =\int_{L^2}\lambda(du_0)\frac{1}{t}\int_0^t P_{\tau}^{\alpha,N}(u_0, B_R^c)d\tau\\
   &\leq \frac{1}{t}\int_{L^2}\lambda(du_0)\int_0^t \frac{\mathbb{E}\lVert u_{\alpha,N}(\tau,u_0) \rVert_{H^{s-1}}^2}{R^2}d\tau ~~~~\text{(according Chebyshev's inequality)}\\
   &\leq\int_{L^2}\lambda(du_0)\frac{\mathbb{E}M(u_0)+\frac{\alpha A^0_N t}{2}}{\alpha t R^2}     \leq \frac{1}{R^2}\left( \frac{\mathbb{E}_{\lambda}\lVert{u_0}\rVert_{L^2}^2}{\alpha t}+\frac{A^0_N}{2} \right).
\end{align*}
Let us choose $\lambda$ the Dirac measure concentrated at 0, $\delta_0$. Then, we obtain $\bar{\lambda}_t(B_R^c)\leq \frac{A^0_N}{2R^2}$, and indeed $B_R$ is compact of $L^2$.\\
Therefore, we have the tightness of $\{\Bar{\lambda_t}\}_{t\geq 0}$, then there is subsequence $ \{t_n\}_{n\in \mathbb{N}}\subset \tau_+$ and measure $\mu_N^{\alpha}$  such that: $\bar{\lambda}_{t_n} \to \mu_N^{\alpha}$ weakly in $P(L^2)$> The classical Bogoliubov-Krylov argument shows that $\mu_N^{\alpha}$ is invariant.\\
\subsection*{Estimates} Let us define $\chi \in C^{\infty}$ be a function such that
\[
\chi(x) =
\begin{cases}
1 & \text{for } x \in [0,1], \\
0 & \text{for } x \in [2,\infty).
\end{cases}
\] and $\chi_R(x)=\chi(\frac{x}{R})$.\\
Since $u \to \mathcal{M}(u)\chi_R(\lVert{u}\rVert_{L^2})$ is continuous and bounded on $L^2$, then we have \\
$$\int_{L^2}\mathcal{M}(u)\chi_R(\lVert{u}\rVert_{L^2}) \bar{\lambda}_{t_n}(du)\leq \int_{L^2}\mathcal{M}(u)\bar{\lambda}_{t_n}(du)=\frac{1}{t_n}\int_0^{t_n}\mathbb{E}\mathcal{M}(u) d\tau\leq \frac{A^0_N}{2}.$$\\
By passing to the limit $t_n \to \infty,R\to \infty$ and using the Fatou's lemma, we obtain: $$\int_{L^2} \mathcal{M}(u) \mu_{N}^{\alpha}(du) \leq \frac{A_N^0}{2},$$ then $$\int_{L^2}\lVert u \rVert_{H^{s-1}}^2 \mu_{N}^{\alpha}(du)\leq \frac{A_N^0}{2}.$$\\
Thus $$\int_{L^2}\lVert{u}\rVert_{L^2}^2 \mu_{N}^{\alpha}(du)\leq \frac{A_N^0}{2} \quad \quad\text{so} ~~~~~\mathbb{E}M(u_0)<\infty.$$
Therefore using the (\ref{estimate1}), and the fact that $\mu_{N}^{\alpha}$ is invariant,\\ we obtain the claim: $$\int_{L^2} \mathcal{M}(u) \mu_{N}^{\alpha}(du) =\frac{A_N^0}{2}\leq \frac{A^0}{2}.$$\\
We have in particular $$\int_{L^2}\lVert{u}\rVert_{L^2}^q \mu_{N}^{\alpha}(du)\lesssim\frac{A_N^0}{2}, \quad \forall q\geq 1.$$\\
Thus by using the fact that we are in finite dimensional, we will have: $$\int_{L^2}E(u)\mu_{N}^{\alpha}(du) <\infty \quad \quad \text{so} ~~~\mathbb{E}E(u_0)<\infty.$$\\
Therefore by using (\ref{estimate2}) and the invariant measure,\\ we will have the claim :$$\int_{L^2} \mathcal{E}(u) \mu_{N}^{\alpha}(du) \leq C.$$
\end{proof}
\begin{proposition}
    For any $R>0$, we have
    \begin{align}\label{estimate5}
    \int_{L^2} \mathcal{M}(u)(1-\chi_R(\lVert{u}\rVert_{L^2}^2)) \mu_{N}^{\alpha}(du) \leq C_1R^{-1} 
    \end{align}
    where $C_1$ is independent of $\alpha,N.$
\end{proposition}
\begin{proof}
    Consider the functional
     $$F_R(u)=\lVert{u}\rVert_{L^2}^2(1-\chi_R(\lVert{u}\rVert_{L^2}^2)).$$
  Applying the Ito formula, we have: 
    \begin{align*}
dF_R + 2\alpha \mathcal{M}(u)(1 - \chi_R (\lVert{u}\rVert_{L^2}^2)) = -2\alpha \mathcal{M}(u)\chi'_R (\lVert{u}\rVert_{L^2}^2) &+ \frac{\alpha}{2} (1 - \chi_R(\lVert{u}\rVert_{L^2}^2)) A_N^0 + 2\chi'_R(\lVert{u}\rVert_{L^2}^2) \sum_{n=0}^Na_n^2(u, e_n)^2\\
&+\frac{\alpha}{2}\lVert{u}\rVert_{L^2}^2\left[ A^0_N\chi'_R(\lVert{u}\rVert_{L^2}^2)+\chi_R''(\lVert{u}\rVert_{L^2}^2)\sum_{n=0}^N a_n^2(u,e_n)^2 \right].
\end{align*}
Using the invariance and (\ref{estimate3}), we obtain:
\[
\mathbb{E}\mathcal{M}(u)(1 - \chi_R(\lVert{u}\rVert_{L^2}^2)) \leq A_N^0 \mathbb{E}(1 - \chi_R(\lVert{u}\rVert_{L^2}^2)) + \frac{C(A^0)}{R}.
\]
Now the Markov inequality and the estimate (\ref{estimate3}) imply
\[
\int_{L^2} (1 - \chi_R(\lVert{u}\rVert_{L^2}^2)) \mu_N^{\alpha} \leq C R^{-1},
\]
 where $ C$ is independent of $(\alpha, N)$.
Overall, we have:
\[
\int_{L^2} \mathcal{M}(u)(1 - \chi_R(\lVert{u}\rVert_{L^2}^2)) \mu_N^{\alpha} \leq C_1 R^{-1},
\]
which is the claim.
\end{proof}

\leavevmode
\subsection{Inviscid limit}
\leavevmode
\begin{proposition}
    Let $N\geq 2$, there exist a measure $\mu_N$ that is invariant under the flow $\phi^t_N$ and satisfies the following estimates:
    \begin{align}\label{estimate6}
\int_{L^2} \mathcal{M}(u) \mu_{N}(du) =\frac{A_N^0}{2}\leq &\frac{A^0}{2}\\ \label{estimate7}
\int_{L^2} \|u\|_{H^s}^2+C_{d,s}\lVert u\rVert_{H^{s^-}}^{3\tilde{k}}(\lVert u\rVert_{H^1}^2+\lVert u\rVert_{L^{2q+2}}^{2q+2})+\frac{1}{q+1}\lVert (-\Delta)^{\frac{s-1}{2}}|u|^{q+1}\lVert^2_{L^2} \mu_N(du) &\leq C;\quad\text{for}\quad{s_{M^d}<s\leq2}  \\ \label{estimate7'}
\int_{L^2} \frac{1}{2}\lVert u\rVert_{H^s}^2+C_{d,s}\lVert u\rVert_{H^{s^-}}^{3\tilde{k}}\lVert u\rVert_{L^{2q+2}}^{2q+2}+ \frac{C_{d,s}}{2}\lVert u\rVert_{L^2}^2\lVert u\rVert_{H^{s^-}}^{3\tilde{k}}\mu_N(du)&\leq C+K;\quad\text{for}\quad\max({s_{M^d},2})<s
    \\ \label{estimate8}
\int_{L^2} \mathcal{M}(u)(1-\chi_R(\lVert{u}\rVert_{L^2}^2)) \mu_{N}(du) &\leq CR^{-1}.
\end{align}
\end{proposition}
\begin{proof}
\leavevmode
\subsection*{Existence}
Thanks to the estimate (\ref{estimate3}), we have the weak compactness of any sequence $(\mu_N^{\alpha} )_{\alpha\in (0,1)}$ with respect to the topology of $L^2$, therefore there exists a subsequence $(\mu_N^{\alpha_k}:= \mu_N^{k})$ , converging to a measure $\mu_N$ on $L^2$. 
\subsection*{Estimates}
Since $\mu_N^\alpha$ and $ \mu_N $ are supported on $E_N$ and we are actually working on a finite dimensional space, then the functions $\mathcal{M}, \mathcal{E}$ are continuous. Hence, the estimates (\ref{estimate7}),\eqref{estimate7'} and (\ref{estimate8}) follow respectively from (\ref{estimate4}) and (\ref{estimate5}), the lower semicontinuity of $\mathcal{M}(u)$ and $\mathcal{E}(u)$.  ~~Let us prove the estimate (\ref{estimate6}), we have 
\[
\frac{A_N^0}{2} - \int_{L^2} (1 - \chi_R(\lVert{u}\rVert_{L^2}^2)) \mathcal{M}(u) \mu_{N}^k (du) \leq \int_{L^2} \chi_R(\lVert{u}\rVert_{L^2}^2) \mathcal{M}(u) \mu_{N}^k (du) \leq \frac{A_N^0}{2}.
\]
Now, by using (\ref{estimate5}), we will have
\[
\frac{A_N^0}{2} - C_1 R^{-1} \leq \int_{L^2} \chi_R(\lVert{u}\rVert_{L^2}^2) \mathcal{M}(u) \mu_{N}^k (du) \leq \frac{A_N^0}{2}.
\]
By letting $k\to \infty$ and $R\to \infty$, we arrive at the claim  
$\int_{L^2} \mathcal{M}(u) \mu_{N}(du) =\frac{A_N^0}{2}\leq \frac{A^0}{2}.$

\subsection*{Invariance}
It is enough to show only the invariance under $\phi^t_N$ for $t>0$; because for $t<0$, we have by using the invariance for positives times, 
$\mu_N(\Gamma)=\mu_N(\phi^t_N\Gamma)=\mu_N(\phi^{2t}_N\phi_N^{-t}\Gamma)=\mu_N(\phi_N^{-t}\Gamma)$ which is what we wanted to show.\\
Now the proof of the invariance for positives times is summarized in the following diagram
\[
\begin{tikzcd}
\mathcal{B}^{k,N*}_t \mu_{N}^k \arrow[r,equal, "\text{(I)}"] \arrow[d, "\text{(III)}"'] & \mu_{N}^k \arrow[d, "\text{(II)}"] \\
\Phi_N^{t*} \mu_N \arrow[r,equal, "\text{(IV)}"'] & \mu_N
\end{tikzcd}
\]
The equality (I) represents the invariance of the measure $\mu_N^k$ under $\mathcal{B}_t^{k,N*}$, the convergence (II) represents the weak convergence of $\mu_N^k$ towards $\mu_N$, the equality (IV) represents the claimed invariance of  $\mu_N$ under $\phi_N^t$ which will follow once we prove the convergence (III) in the weakly topology of $L^2$.

We are going now to prove the convergence (III). For that, let $f: L^2 \to \mathbb{R}$ be a lipschitz function that is also bounded by $1$, we have \\
\begin{align*}
    (\mathcal{B}_t^{k,N*}\mu_N^k,f)-(\phi_N^{t*}\mu_N,f)&=(\mu_N^k,\mathcal{B}_t^{k,N}f)-(\mu_N,\phi_N^tf)\\
    &=(\mu_N^k,\mathcal{B}_t^{k,N}f-\phi_N^tf)-(\mu_N-\mu_N^k,\phi_N^tf)\\
    &= A_1-A_2.
\end{align*}
By using the fact that $\phi_N^t$ is Feller, we see that $A_2 \to 0$ as $k\to \infty$.\\
On the other hand, by using the boundedness property of $f$, we have \\
$$\abs{A_1}\leq \int_{B_R(L^2)}\abs{\phi_N^tf(u_0)-\mathcal{B}_t^{k,N}f(u_0)}\mu_N^k(du_0)+2\mu_N^k(L^2\setminus B_R(L^2))=A_3+A_4.$$
We have 
\begin{align*}
    A_3&= \int_{B_R(L^2)}\left|\left(\int_{L^2}f(v)P_t^{k,N}(u_0,dv)\right)-\phi_N^tf(u_0) \right|\mu_N^k(du_0)\\
    &=\int_{B_R(L^2)}\left|\left( \int_{\Omega}f(u_k(t,P_Nu_0))d\mathbb{P}\right)- f(\phi_N^tP_Nu_0) \right|\mu_N^k(du_0)\\
    &\leq \int_{B_R(L^2)}\left( \int_{\Omega}\left|(f(u_k(t,P_Nu_0))- f(\phi_N^tP_Nu_0))\right|d\mathbb{P}\right) \mu_N^k(du_0)\\
    &\leq  \int_{B_R(L^2)}\left( \int_{S_r}\left|(f(u_k(t,P_Nu_0))- f(\phi_N^tP_Nu_0))\right|d\mathbb{P}\right) \mu_N^k(du_0)\\
    &\quad \quad+ \int_{B_R(L^2)}\left( \int_{S_r^c}\left|(f(u_k(t,P_Nu_0))- f(\phi_N^tP_Nu_0))\right|d\mathbb{P}\right) \mu_N^k(du_0)\\
    &\leq C_f\int_{B_R(L^2)}\mathbb{E}\left( \lVert u_k(t,P_Nu_0)-\phi_N^tP_Nu_0\rVert_{L^2}\mathbf{1}_{S_r}\right)\mu_N^k(du_0)+2\int_{B_R(L^2)}\mathbb{E}\mathbf{1}_{S_r^c}~~\mu_N^k(du_0)
\end{align*}
by using the Lipschitz property.\\
Now let us consider for $r>0$,
$$ S_r=\left\{ \omega\in \Omega|\max\left(\left|\sqrt{\alpha_k}\sum_{n=0}^Na_n\int_0^t(u,e_n)d\mathcal{B}_n(\tau) \right|,\lVert z_k\rVert_{L^2}\right)\leq r\sqrt{\alpha_k t} \right\}.$$
Let us compute $\mathbb{E} \mathbf{1}_{S_r^c}$. We have \\
$$\mathbb{E}\left|\sqrt{\alpha_k}\sum_{n=0}^Na_n\int_0^t(u,e_n)d\mathcal{B}_n(\tau) \right|^2=\alpha_k\sum_{n=0}^Na_n^2\int_0^t\mathbb{E} (u,e_n)^2d\tau\leq \alpha_kt A_0\mathbb{E}\lVert u\rVert_{L^2}^2\leq C\alpha_kt .$$\\
We have also $$\mathbb{E}\lVert z_k\rVert^2_{L^2}\leq C\alpha_kt. $$\\
According to the Chebyshev inequality, we have:\\
$$\mathbb{E} \mathbf{1}_{S_r^c}=\mathbb{P}\left\{ w| \max\left(\left|\sqrt{\alpha_k}\sum_{n=0}^Na_n\int_0^t(u,e_n)d\mathcal{B}_n(\tau) \right|,\lVert z_k\rVert_{L^2}\right)\geq r\sqrt{\alpha_k t}\right\}\leq \frac{C\alpha_k t}{r^2\alpha_kt}=\frac{C}{r^2}.$$
We need to prove now the following statement 
\begin{lemma}
    We have for any $R>0$ and $r>0$ 
    \begin{align}
\sup_{u_0 \in B_R(L^2)} \mathbb{E}\left( \lVert \phi_N^t P_N u_0 - u_k(t, P_N u_0)\rVert_{L^2}\mathbf{1}_{S_r} \right) \to 0 \text{ as } k \to \infty.
\end{align}
\end{lemma}

\begin{proof}
   Let us recall the equations
  \begin{equation*}
\left\{
\begin{aligned}
   \partial_t u &= i\left[ \Delta u - P_N \left( |u|^{2q} u \right) \right], \\
   du_{\alpha_k} &= i\left[ \Delta u_{\alpha_k} - P_N \left( |u_{\alpha_k}|^{2q} u_{\alpha_k} \right) - \alpha_k ((-\Delta)^{s-1}u_{\alpha_k}+C_{d,s}\lVert u_{\alpha_k} \rVert_{H^{s^-}}^{3\tilde{k}}u_{\alpha_k}) \right] dt + \sqrt{\alpha_k} d\mathcal{W}^N, \\
   u_{\alpha_k} &= v_k + z_{\alpha_k}, \\
   dz_{k} &= (i\Delta z_{\alpha_k}-\alpha_k(-\Delta)^{s-1}z_{\alpha_k})dt + \sqrt{\alpha_k} \sum_{n=1}^N a_n e_n d\mathcal{B}_n(t), \\
   \partial_t v_k &= i \left( \Delta v_k - P_N \left( |v_k + z_{\alpha_k}|^{2q}(v_k + z_{\alpha_k}) \right) \right) - \alpha_k( (-\Delta)^{s-1}v_k+C_{d,s}\lVert v_k + z_{\alpha_k}\rVert_{H^{s^-}}^{3\tilde{k}}(v_k + z_{\alpha_k})).
\end{aligned}
\right.
\end{equation*}
Let us set $w_k=u-v_k=\phi_N^tP_Nu_0-v_k(t,P_Nu_0)$. We have $$\partial_tw_k=i\left[\Delta w_k-P_N( |u|^{2q} u- |v_k + z_{\alpha_k}|^{2q}(v_k + z_{\alpha_k}))  \right]+ \alpha_k( (-\Delta)^{s-1}v_k+C_{d,s}\lVert v_k + z_{\alpha_k}\rVert_{H^{s^-}}^{3\tilde{k}}(v_k + z_{\alpha_k})).$$
Since $$\mathbb{E}\lVert z_{\alpha_k}\rVert_{L^2}\to 0 ~~\text{as}~~ k \to \infty ~~~\forall ~u_0\in B_R(L^2),$$ we just need to prove that $$\sup_{u_0 \in B_R(L^2)} \mathbb{E}\left( \lVert w_k\rVert_{L^2}\mathbf{1}_{S_r} \right) \to 0 \text{ as } k \to \infty.$$\\
Let us define $f_{p-1}$ and $g_{p-1}$ such that $|u|^{2q} u-|v_k + z_{k}|^{2q}(v_k + z_{k})=w_kf_{2q}(u,v_k)-g_{2q}(v_k,z_k)z_k$\\
where $f_{2q}$ and $g_{2q}$ are polynomials of degree $2q$ in the given variables.\\
By taking the inner product with $w_k$, we obtain
\begin{align*}
    \partial_t \lVert w_k\rVert_{L^2}^2&\leq 2\lVert w_k\rVert_{L^2}^2(1+\lambda_N^2+\lVert f_{2q}(u,v_k)\rVert_{L^{\infty}_{t,x}})+2\lVert z_k\rVert_{L^2}^2\lVert g_{2q}(v_k,z_k)\rVert_{L^{\infty}_{t,x}}^2\\
    & +\alpha_kC_0(N)\left(\frac{\lVert v_k\rVert^2_{L^2}}{2}+\frac{\lVert w_k\rVert^2_{L^2}}{2}+\lVert v+z_{k}\rVert_{L^2}^{3\tilde{k}}\left(\frac{\lVert v+z_k\rVert_{L^2}^2}{2}+\frac{\lVert w_k\rVert_{L^2}^2}{2}\right)\right) \\
    &\leq C_1(N)\lVert w_k\rVert_{L^2}^2\left(1+\lambda_N^2+\lVert u\rVert_{L^{\infty}_{t}{L^2_x}}^{2q}+\lVert v_k\rVert_{L^{\infty}_{t}{L^2_x}}^{2q}+\alpha_k+\alpha_k\lVert u_k\rVert_{L^{\infty}_{t}{L^2_x}}^{3\tilde{k}}\right)\\
    &+C_2(N)\lVert z_k\rVert_{L^2}^2\left(\lVert v_k\rVert_{L^{\infty}_{t}{L^2_x}}^{4q}+\lVert z_k\rVert_{L^{\infty}_{t}{L^2_x}}^{4q}+\alpha_k\left(\lVert v_k\rVert_{L^{\infty}_{t}{L^2_x}}^{2}+\lVert u_k\rVert_{L^{\infty}_{t}{L^2_x}}^{3\tilde{k}+2}\right)\right).
\end{align*}
By using the Gronwall's lemma and the fact that $w_k(0)=0$, we have $\mathbb{P}.a.e$
\begin{align}\label{estimate10}
\lVert w_k\rVert^2_{L^2}\leq C_3(N,\alpha_k,u_0,q,,\tilde{k})e^{C_1(N)\int_0^t\left(1+\lambda_N^2+\lVert u\rVert_{L^{\infty}_{t}{L^2_x}}^{2q}+\lVert v_k\rVert_{L^{\infty}_{t}{L^2_x}}^{2q}+\alpha_k+\alpha_k\lVert u\rVert_{L^{\infty}_{t}{L^2_x}}^{3\tilde{k}}\right)d\tau}\left(\int_0^t\lVert z_k\rVert^2_{L^2}d\tau \right).
\end{align}
We know that $$\lim_{k\to 0}\sup_{t\in [0,T]}\lVert z_k\rVert_{L^2}=0\quad \mathbb{P}.a.e.$$ Now, by using the Ito formula on the stochastic solution $u_k$, and using the fact $\alpha_k\leq 1$ and we are on the $S_r$, we arrive at $$\lVert u_k\rVert_{L^2}^2\leq \lVert P_Nu_0\rVert_{L^2}^2+C(r,N)t.$$\\
After that, it is easy to see that on $S_r$, $$\lVert w_k\rVert_{L^2}\leq \lVert v_k\rVert_{L^2}+\lVert z_k\rVert_{L^2} \leq \lVert u_k\rVert_{L^2} +2\lVert z_k\rVert_{L^2}\leq \lVert u_0\rVert_{L^2}+3C(r,N)T.$$
\begin{equation*}
\begin{cases}
    \sup_{u_0 \in B_R} \lVert w_k\rVert_{L^{\infty}_t L^2_x} \mathbf{1}_{S_r} \leq R + 3C(R,N)T, \\
    \sup_{k \geq 1} \sup_{u_0 \in B_R} \lVert w_k\rVert_{L^{\infty}_t L^2_x} \mathbf{1}_{S_r} \leq R + 3C(R,N)T.
\end{cases}
\end{equation*}
Thus, using the estimate (\ref{estimate10}) and the  $L^2$ conservation of the deterministic solution\\ $\lVert u(t)\rVert_{L^2}=\lVert P_Nu_0\rVert_{L^2}$, we have$$\sup_{u_0 \in B_R} \lVert w_k\rVert_{L^{\infty}_t L^2_x}^2 \mathbf{1}_{S_r} \leq A(R,N,r,T)\|z_k\|_{L^1_tL^2_x}.$$\\
We obtain that $$\lim_{k\to 0}\sup_{u_0\in B_R}\lVert w_k\rVert_{L^{\infty}_t L^2_x}^2\mathbf{1}_{S_r}=0$$\\
So $$\mathbb{E}\sup_{u_0\in B_R}\lVert w_k\rVert_{L^{\infty}_t L^2_x}^2\mathbf{1}_{S_r}\to 0\quad \text{as} \quad k\to \infty$$
Now using the fact that $\forall u_0\in B_R$, we have $\lVert w_k(t,P_Nu_0)\rVert_{L^2}\mathbf{1}_{S_r} \leq \sup_{u_0\in B_R}\lVert w_k(t,P_Nu_0)\rVert_{L^{\infty}_t L^2_x}\mathbf{1}_{S_r}$\\
We will have therefore $$\sup_{u_0 \in B_R(L^2)} \mathbb{E}\left( \lVert w_k\rVert_{L^2}\mathbf{1}_{S_r} \right) \to 0 \text{ as } k \to \infty.$$
Which is the claim.
 
\end{proof}
By letting $R,r,k \to \infty$ and using the lemma, we have $A_3, A_4 \to 0$. This completes the proof of invariance.
\end{proof}

\leavevmode
\subsection{Uniform control on the finite-dimensional dynamics}
\leavevmode

Let $B_a=\{u\in L^2| ~~\lVert u\rVert_{L^2}\leq a\}$
where the number $a > 0$  is arbitrary small.\\
Let us set  $L_{a}^2 = L^2 \setminus B_a, ${ and } $E_{N}^{a} = \{ u \in E_{N} \mid \lVert u\rVert_{L^2} > a \}.$
We have 
\begin{align}
   \int_{E_N^{a}} \lVert u\rVert_{H^{s^-}}^{3\tilde{k}}d\mu_N\leq C. 
\end{align}
Indeed, \begin{align*}
    \int_{E_N^{a}}\lVert u\rVert_{H^{s^-}}^{3\tilde{k}}\mu_N(du)=\int_{L^2_a}
\lVert u\rVert_{H^{s^-}}^{3\tilde{k}}\mu_N(du)\leq \int_{L^2_a}\frac{\lVert u\rVert_{L^2}^2}{a^2}\lVert u\rVert_{H^{s^-}}^{3\tilde{k}}\mu_N(du)\leq C:=C(a).
\end{align*}
\begin{proposition}\label{globalization1}
Let $s_{M^d}< s' \leq s^-$, let $\epsilon>0$, let $a>0$ and let $N \geq 0$. There exists a constant $C=C(a)  > 0$ and $\tilde{k}:=\tilde{k}(\epsilon)$ such that:
for any $i \in \mathbb{N}^*$, there is a set $\Sigma^{i}_{N,s'}$ satisfying 
\begin{align}\label{Bound_level-ens}
 \mu_N\left(E_N^{a} \setminus \Sigma^{i}_{N,s'}\right) \leq C i^{-2\tilde{k}},  
\end{align}
and having the property that:
$\forall\, u_{0,N} \in \Sigma^{i}_{N,s'}:$
\begin{align}\label{Est_individ}
    \forall t \in \mathbb{R}, \quad \|\phi_N^t u_{0,N}\|_{H^{s'}} \leq C(\|u_0\|_{H^{s}})((1+|t|))^{\epsilon}.
\end{align} 
\end{proposition}

\begin{proof}
Without loss of generality, we will consider $t\geq0$
Let us define for $ j\geq 1$, the set $$B^{i,j}_{N,s'}=\left\{u_{0,N}\in E_N|~ \|u_{0,N}\|_{H^{s'}}\leq ij\right\}.$$ and let $T_0\sim(ij)^{\frac{-2q}{\gamma}}$ (with $\gamma <1$) be the local time existence:
\begin{align}\label{Evol_ball}
    \forall t\in [0,T_0],~ \phi^t_N B^{i,j}_{N,s'}\subset \left\{u\in E_N\ | ~\|u\|_{H^{s'}}\leq 2C~ij\right\}.
\end{align}
Let us set $$\Sigma^{i,j}_{N,s'} = \bigcap_{l=0}^{\left\lfloor \frac{j^{\tilde{k}}}{T_0} \right\rfloor} \phi^{-lT_0}_N\left(B^{i,j}_{N,s'}\right).$$
 By using the invariance of $\mu_N$ under $\phi^N_t$, we see that
\begin{align}
\mu_N(E_N^{a}\backslash\Sigma_{N,s'}^{i,j}) \leq \sum_{l=0}^{[\frac{j^{\tilde{k}}}{T_0}]} \mu(E^a_N\backslash\Sigma^{i,j}_{N,s'}).
\end{align}
Since $s'\leq s^-$, then $\mathbb{E}_{\mu_N}\|u\|^{3\tilde{k}}_{H^{s'}}\leq C:=C(a)$ and $T_0\sim(ij)^{\frac{-2q}{\gamma}}$, therefore by using Chebyshev's inequality, we obtain 
\[
\mu_N(E_N^{a} \backslash \Sigma_{N,s'}^{i,j}) 
\leq C_1 \left(\left\lfloor \frac{j^{\tilde{k}}}{T_0} \right\rfloor + 1\right)(i \cdot j)^{-3\tilde{k}} \leq C j^{\tilde{k}} (i \cdot j)^{\frac{2q}{\gamma}} (i \cdot j)^{-3\tilde{k}} 
\leq C j^{\tilde{k}} (i \cdot j)^{\tilde{k}} (i \cdot j)^{-3\tilde{k}} 
\leq C i^{-2\tilde{k}} j^{-\tilde{k}}.
\]
Notice that the series $\sum_{j\geq 1}j^{-\tilde{k}}$ converges and let us set $$\Sigma^{i}_{N,s'}=\bigcap_{j\geq 1}\Sigma^{i,j}_{N,s'}.$$ 
Then we obtain $$\mu_N\left(E_N^{a} \setminus \Sigma^{i}_{N,s'}\right)\leq Ci^{-2\tilde{k}},$$
where $C$ does not depend on $N$.\\
We see also that:
$$\forall u_{0,N} \in \Sigma^{i,j}_{N,s'},~ \forall t\leq j^{\tilde{k}},~ \|\phi^t_N u_{0,N}\|_{H^{s'}}\leq 2C~ij$$ In fact, we have $\forall t\leq j^{\tilde{k}}, t=lT_0+\tau$ where $l\in[0,\lfloor\frac{j^{\tilde{k}}}{T_0}\rfloor]$ and $\tau \in [0,T_0[$. By definition of $\Sigma^{i,j}_{N,s'}$, $u_{0,N} $ can be written as $\phi_N^{-l_1T_0}w$ with a fixed integer $l_1\in[0,[\frac{j^{\tilde{k}}}{T_0}]]$ and $w \in B^{i,j}_{N,s'} $.\\ 
Then $\phi_N^tu_{0,N}=\phi_N^{\tau}\phi_N^{lT_0}u_{0,N}=\phi_N^{\tau}w$.

We then obtain, with the use of \eqref{Evol_ball}, $$\|\phi^t_N u_{0,N}\|_{H^{s'}}\leq 2C ij \quad \forall t\leq j^{\tilde{k}}.$$ \\
 Let $t>0$, there is $j\geq 1$ such that $j^{\tilde{k}}-1\leq |t|\leq j^{\tilde{k}}$. Then $$j\leq (1+\abs{t})^{\frac{1}{\tilde{k}}}.$$
We arrive at $$ \forall u_0 \in \Sigma^{i}_{N,s'}, \forall t \in \mathbb{R} \quad \|\phi^t_N u_0\|_{H^{s'}}\leq 2Ci(1+|t|)^{\frac{1}{\tilde{k}}}\leq 2C i(1+|t|)^\epsilon. $$
This completes the proof.
\end{proof}
\leavevmode
\subsection{Construction of the statistical ensemble}
The estimates (\ref{estimate7},~\ref{estimate7'}) combined with the Prokhorov theorem establish the existence of a measure $\mu$ as a weak limit of a subsequence $\mu^{N_k}\subset (\mu^N)$, and a simple argument show that $\mu(H^s)=1$.\\
We employ the Skorokhod representation theorem to obtain a probability space still denoted $(\Omega,\mathcal{F},\mathbb{P})$ on which are defined the random variables $u_{N_k}$ and $u$ satisfying the following.
\begin{itemize}
    \item $u$ is distributed by $\mu$ and $\forall N_k, u_{N_k}$ is distributed by $\mu_{N_k}$;
    \item $u_{N_k}$ converges to $u_{}$ almost surely in $H^{s'}$.
\end{itemize}

Let us introduce the set\\
\begin{align}
\Sigma^{i}_{s'} = \left\{ u \in H^{s'} \mid  \exists (u_{N_k}), \ u_{N_k} \to u \ \text{as} \ k \to \infty, \ \text{and} \ u_{N_k} \in \Sigma^{i}_{N_k,s'} \right\}.
\end{align}
 \\ Let us introduce the statistical ensemble
\begin{align}
\Sigma_{s'} = \bigcup_{i \geq 1} \Sigma^{i}_{s'}.
\end{align}
\begin{lemma}\label{Lemma_full_mu}
    We have that
    \begin{align}
        \mu(\Sigma_{s'})=1.
    \end{align}
\end{lemma}
\begin{proof}
In the line of Skorokhod sequences above, we can find a full probability set $\tilde{\Omega}$ such that the convergence holds for all $\omega\in\tilde{\Omega}.$ Define the sets
\begin{align}
    \Bbb A_{N,{s'}} &=\{\omega\in\tilde{\Omega}\ |\ u_N^\omega\in \Sigma^i_{N,{s'}}\}\\
    \Bbb A_{s'} &=\{\omega\in\tilde{\Omega}\ |\ u^\omega\in \Sigma^i_{s'}\}.
\end{align}
We remark that, by the construction of $\tilde{\Omega}$, for all $\omega\in\tilde{\Omega}$, the sequence $\{u_{N_k}^\omega\}$ converges.
Let us prove that
\begin{align}\label{Eq_limsup}
        \limsup_N\Bbb A_{N,{s'}}\subset \Bbb A_{s'}.
    \end{align}
If $\omega$ belongs to an infinite number of $\Bbb A_{N_k,{s'}}$, the corresponding $L^2$ elements form a sequence $\{u_{N_k}^\omega\}$ whose $kth$ element, for any $k$, belong to $\Sigma^i_{N_k,{s'}}$. Recalling the convergence established above, we obtain the existence of an element $u^\omega$ as its limit. Therefore $\omega$ belongs to $\Bbb A_{s'}$. We arrive at \eqref{Eq_limsup}.
\\Next, with the use of the inclusion \eqref{Eq_limsup}, the inequality \eqref{Limsup_nu} below and the bound \eqref{Bound_level-ens}, let us write 
\begin{align}
    \mu(\Sigma^i_{s'})=\Bbb P(u^{-1}(\Sigma^i_{s'}))=\Bbb P(\Bbb A_{s'})\geq \Bbb P(\limsup_N\Bbb A_{N,{s'}}) &\geq \Bbb \limsup_{N\to\infty}\Bbb P(\Bbb A_{N,{s'}})\\
    &= \limsup_{N\to\infty}\Bbb P (u_N^{-1}(\Sigma^i_{N,{s'}}))\\
    &=\limsup_{N\to\infty}\mu_N(\Sigma^i_{N,{s'}})\geq 1-i^{-2\tilde{k}}.
\end{align}
Now, since $\mu$ is finite and the sequence $(\Sigma^i_{s'})_i$ is non decreasing, we have
\begin{align}
   1\geq \mu(\Sigma_{s'})=\lim_{i\to\infty}\mu(\Sigma^i_{s'})\geq \lim_{i\to\infty}(1-Ci^{-2\tilde{k}})=1.
\end{align}
This finishes the proof.
\end{proof}
\begin{lemma}
    Let $(E,\mathcal{F},\nu)$ be a measure space, with $\nu(E)<\infty$. Then, for all sequences of measurable sets $(E_k)$, we have
    \begin{align}\label{Limsup_nu}
        \nu(\limsup_{k\to\infty} E_k)\geq \limsup_{k\to\infty}\nu(E_k).
    \end{align}
\end{lemma}
\begin{proof}
We have that the sequence $F_k:=\cup_{j\geq k}E_j$ is non increasing in $k$. Then
    \begin{enumerate}
        \item $\nu(F_k)\geq \sup_{j\geq k}\nu(E_j)$,
        \item $\nu(\cap_{k\geq 0}F_k)=\inf_{k\geq 0}\nu(F_k)$.
    \end{enumerate}    
    Putting these together, we obtain
    \begin{align}
         \nu(\limsup_{k\to\infty} E_k)\geq\inf_{k\geq 0}\sup_{j\geq k}\nu(E_j)=\limsup_{k\to\infty}\nu(E_k),
    \end{align}
    which was the claim.
\end{proof}

\begin{proposition}
  Let $M^d $ be a Riemannian compact manifold of dimension $d$.
For any $s_{M^d}< s$, for every $u_0 \in \Sigma_{s'}$, the solution $\phi^tu_0$ of (\ref{Intro_NLS}) given by Theorem \ref{local well posedness} is global in $H^s$. And 
\begin{equation}
\|\phi^t u_0\|_{H^{s'}}\leq C(\lVert u_0\rVert_{H^s})(1+|t|)^\epsilon \quad \forall t\in \mathbb{R}.  
\end{equation}
\end{proposition}

\begin{proof}
    We just need to prove the globalization lemma assumptions.\\
    Let $ u_0 \in \Sigma_{s'}$,
    then $u_0 \in \Sigma^{i}_{s'}$ for some $i$. By definition of $\Sigma^{i}_{s'},  \exists (u_{0,N_k}), \ u_{0,N_k} \to u_0 \ \text{as} \ k \to \infty, \ \text{and} \ u_{0,N_k} \in \Sigma^{i}_{N_k,s'}. $\\
    Hence, we have
    \begin{enumerate}
        \item the existence of the sequence that is the first condition of the (globalization) lemma \ref{globalization lemma}; 
        \item and using  Proposition (\ref{globalization1}), we have that (since $ u_{0,N_k} \in \Sigma^{i}_{N,s'}$)  $$ 
\|\phi^t_{N_k} u_0\|_{H^{s'}}\leq C(\lVert u_0\rVert_{H^s})(1+|t|)^\epsilon \quad \forall t\in \mathbb{R}.  
$$
    \end{enumerate}
Hence we can invoke the globalization lemma \ref{globalization lemma} to obtain the claim.
\end{proof}
Now let us consider an increasing sequence \(  (s'_n)_{n \in \mathbb{N}} \) such that $ s_{M^d}<s'_n $ and \( \lim_{n \to \infty} s'_n = s^-\) and let us set our statistical ensemble: $$\Sigma=\Sigma_{s,q,\epsilon}= \bigcap_{n\in \mathbb{N}}\Sigma_{s'_n}.$$
Since each of the sets $\Sigma_{s_n'}$ is of full measure, we have that
\begin{align}
    \mu(\Sigma)=1.
\end{align}
\begin{remark}
  We can summarize all the analysis that we have been done above in the following: For any $u_0\in\Sigma$, the associated solution $\phi^tu_0$ provided by Proposition \ref{local well posedness} is global in time, and\\ $
            \|\phi^tu_0\|_{H^{s^-}}\leq C(\|u_0\|_{H^{s}})(1+|t|)^{\epsilon}\quad \forall t\in \mathbb{R}.$
   \end{remark}
  \leavevmode
\subsection{Further properties of $\Sigma$ and $\mu$}
\begin{proposition}
    The following holds:
\begin{enumerate}
 \item The distribution of functional $u\mapsto\|u\|_{L^2}$ under $\mu$ is absolutely continuous w.r.t. the Lebesgue measure on $\R$. In particular $\mu$ doesn't have an atom.
\item For any $s_{M^d} < s$, the measure \( \mu \) is invariant under \( \phi^t \).\\
\item For any $ s_{M^d} < s' \leq s^-$, any $s_{M^d}< s_1 < s' $, for every \( t \in \mathbb{R} \), there is \( i_1 \in \mathbb{N}^* \) such that for any \( i \in \mathbb{N}^* \), if \( u_{0,N} \in \Sigma^i_{N,s'} \), then we have \( \phi_N^t u_{0,N} \in \Sigma^{ii_1}_{N,s_1} \).\\
\item  The flow \( \phi^t \) satisfies
\( \phi^t \Sigma = \Sigma \), for any \( t \in \mathbb{R} \).
\end{enumerate}
\end{proposition}
\begin{proof}
\leavevmode
    \begin{enumerate}
 \item The proof uses an argument of Shirikyan \cite{shirikyan2011local}, and follows Theorem $9.1$ in Sy \cite{sy2021almost}.\\
 \leavevmode
 \item Let us prove the invariance of $\mu$ under the flot $\phi^t.$\\
  The Ulam's theorem ensures that $\mu$ is regular: for any \( S \in \text{Bor}(H^s) \)
\[
\mu(S) = \sup\{\mu(K), K \subset S \text{ compact} \}.
\]
Therefore, it will just suffices to prove invariance for compact sets. Indeed, we then obtain, for any \( t \),
\[
\mu(\phi^{-t}S) = \sup\{\mu(K), K \subset \phi^{-t}S \text{ compact} \} = \sup\{\mu(\phi^t K), K \subset \phi^{-t}S \text{ compact} \} 
\]
\[
= \sup\{\mu(\phi^t K), \phi^t K \subset S, K \text{ compact} \} \leq \sup\{\mu(C), C \subset S \text{ compact} \} = \mu(S),
\]
where we used the fact that \( \phi^t \) is continuous in space, therefore it transforms compact sets into compact sets.

Using the inequality above, we also have for any \( t \) that
\[
\mu(S) = \mu(\phi^{-t} \phi^t S) \leq \mu(\phi^t S).
\]
Since \( t \) is arbitrary, we then obtain the invariance.

Now we claim that it also suffices to show the invariance only on a fixed interval \( [-\tau, \tau] \), where \( \tau > 0 \) can be as small as we want. Indeed, for \( \tau \leq t \leq 2\tau \), one has
\[
\mu(\phi^{-t} K) = \mu(\phi^{-\tau} \phi^{t-\tau} K) = \mu(\phi^{t-\tau} K) = \mu(K) \quad (\text{using that } 0 \leq t - \tau \leq \tau),
\]
and for greater values of \( t \), we can iterate. A similar argument works for negative values of \( t \).

Our proof is then reduced to showing invariance for compact sets on a small time interval. Therefore, it suffices to show it on the balls of \( H^s \). Below is the idea of the proof:
\[
\begin{tikzcd}
\phi_N^{t*} \mu_{N_k} \arrow[r,equal, "\text{(I)}"] \arrow[d, "\text{(III)}"'] & \mu_{N_k} \arrow[d, "\text{(II)}"] \\
\Phi^{t*} \mu \arrow[r,equal, "\text{(IV)}"'] & \mu
\end{tikzcd}
\]
We are going to prove the converge (III) to conclude the invariance (IV).\\
Let \( f \in C_b(H^s) \), supported on a ball \( B_R(H^s) \). Assume that \( f \) is Lipschitz in the topology of \( H^{s'} \), \( s' < s \). Let \( \tau \) be the associated time existence provided by Proposition 2.3. Then for \( t < \tau \), we have
\begin{align*}
( \Phi^{t*}_N\mu_{N_k} , f) - (\Phi^{t*} \mu , f) &= (\mu_{N_k} , \Phi_N^t f) - (\mu , \Phi^t f)\\
&= (\mu_{N_k}, \Phi_N^tf - \Phi^tf) - (\mu - \mu_{N_k}, \Phi^tf) = A_1 - A_2.
\end{align*}
By the continuity property of \( \phi^t \), we have that \( \Phi^t f \in C_b(H^s) \). Then by weak convergence of \( \mu_{N_k} \) to \( \mu \) on \( H^{s'} \), we have that \( A_2 \to 0 \) as \( N \to \infty \).

Now using the Lipschitz property of \( f \), we have thanks the local uniform convergence,
\[
|A_1| \leq C_f \sup_{u_0 \in B_R(H^s)} \| \phi_N^t(P_Nu_0) - \phi^t(u_0) \|_{H^{s'}} \mu_N (B_R(H^{s'}))
\]
\[
\leq C_f \sup_{u_0 \in B_R(H^s)} \| \phi_N^t(P_Nu_0) - \phi^t(u_0) \|_{H^{s'}} \to 0, \text{ as } N \to \infty.
\]
We obtain the claim. \\
\item Let us fix \( s' \in (s_{M^d}, s^-] \). Let \( u_{0,N} \in \Sigma_{N,s'}^{i} \), and $t\in \mathbb{R}$. Without loss of generality, assume \( t > 0 \) then for any \( j \geq 1 \), we have
\[
\|\phi^{t_1}_N u_{0,N} \|_{H^{s'}} \leq 2Ci j, \quad t_1 \leq j^{\tilde{k}}.
\]

Let $i_1 := i_1(t) \text{ be such that for every } j \geq 1, j^{\tilde{k}} + t \leq (ji_1)^{\tilde{k}}. \text{ We then have }$\\ $$\|\phi_N^{t_1 + t} u_{0,N} \|_{H^{s'}} \leq 2C~ i ( j i_1), \quad t_1 \leq j^{\tilde{k}}.$$
Now, thanks to Proposition \ref{globalization1}, we have, for every $ u_{0,N} \in \Sigma_{N,s'}^{i}$
\[
\lVert u_{0,N} \rVert_{L^2} \leq \| u_{0,N} \|_{H^{s'}} \leq 2Ci,
\]
therefore, since the \( L^2 \)-norm is preserved, we have, for every \( u_{0,N} \in \Sigma_{N,s'}^{i} \), that
\[
\lVert \phi_N^{t_1 + t} u_{0,N} \rVert_{L^2} \leq 2Ci.
\]
Hence for every \(  s_{M^d}< s_1 < s' \), we use an interpolation to see that there is \( \theta \in (0, 1) \) such that for all \( t_1 \leq j^{\tilde{k}} \),
\[
\|\phi_N^{t + t_1} u_{0,N} \|_{H^{s_1}} \leq \|\phi_N^{t + t_1} u_{0,N} \|_{L^2}^{1 - \theta} \|\phi_N^{t + t_1} u_{0,N} \|^\theta_{H^{s'}} \leq  ~(2C)^{1 - \theta} (2C)^\theta  i^{1 - \theta} (i( j i_1))^\theta
\]
\[
\leq  ~  (i i_1)j.
\]
The last inequality above follows from the fact that  $i_1$ can be taken so large that $2Ci_1^\theta\leq i_1$ (since $\theta<1$). We already have that $j^\theta\leq j$. The claim follows.\\\\

\item Since any \( \Sigma_{s'} \) is of full \( \mu \)-measure and the intersection is countable, we obtain the first statement.

To prove the second statement, let us take \( u_0 \in \Sigma \), then \( u_0 \) belongs to each \( \Sigma_{s'} \), \( s' \in l \).

Firstly, let us consider \( u_0 \in \Sigma_{s'}^{i} \). Therefore \( u_0 \) is the limit of a sequence \( (u_{0,N}) \) such that \( u_{0,N} \in \Sigma_{N,s'}^{i} \) for every \( N \).

Now, from the above statement, there is \( i_1 := i_1(t) \) such that \( \phi^t_N (u_{0,N}) \in \Sigma^{ii_1}_{N,s_1} \).\\ By using the convergence $(\lim_{N \to \infty} \| \phi^t u_0 - \phi^t_N u_{0,N} \|_{H^{s_1}} = 0)$, we see that \( \phi^t(u_0) \in \Sigma^{ii_1}_{s_1} \).
We conclude that
\[
\phi^t (\Sigma_{s'}^{i}) \subset \Sigma_{s_1}^{ii_1}  \subset \Sigma_{s_1}.
\]
It follows that \( \phi^t \Sigma \subset \Sigma \).

Now, let \( u \) be in \( \Sigma \), since \( \phi^t \) is well-defined on \( \Sigma \), we can set \( u_0 = \phi^{-t} u \), we then have \( u = \phi^t u_0 \) and hence \( \Sigma \subset \phi^t \Sigma \). 
 \end{enumerate}
This completes the proof.
 \end{proof}
\bibliography{ref}
\bibliographystyle{abbrv}
\end{document}